\documentclass[reqno, letterpaper, oneside]{article}
\usepackage{amsmath,amsfonts,amssymb,amsthm}
\usepackage[final]{graphicx}

\usepackage{xpunctuate} 
\usepackage[square,numbers]{natbib}

\usepackage{enumitem}

\usepackage{bbm}
\usepackage{geometry}

\numberwithin{equation}{section}

\newtheorem{theorem}{Theorem}[section]
\newtheorem{lemma}[theorem]{Lemma}
\newtheorem{corollary}[theorem]{Corollary}
\newtheorem{claim}[theorem]{Claim}

\newtheorem{remarkttt}[theorem]{Remark}
\newtheorem{problem}[theorem]{Problem}

\theoremstyle{definition}
\newtheorem{remarkqqq}[theorem]{Remark}
\newenvironment{remark}{\begin{remarkqqq}}
  {\hfill\qedsymbol\end{remarkqqq}}
\newtheorem{exampleqqq}[theorem]{Example}

\newcommand{\refT}[1]{Theorem~\ref{#1}}

\newcommand{\refL}[1]{Lemma~\ref{#1}}
\newcommand{\refR}[1]{Remark~\ref{#1}}

\newcommand{\refCr}[1]{Corollary~\ref{#1}}
\newcommand{\refS}[1]{Section~\ref{#1}}

\newcommand{\refApp}[1]{Appendix~\ref{#1}}

\newcommand\E{{\mathbb E}}
\renewcommand\Pr{{\mathbb P}}
\newcommand\dto{\overset{\mathrm{d}}{\to}}

\newcommand{\cG}{{\mathcal G}}

\newcommand{\cH}{{\mathcal H}}

\newcommand{\cS}{{\mathcal S}}

\newcommand\Var{\operatorname{Var}}

\newcommand\G{G} 
\newcommand\dn{\mathbf{d}}
\newcommand\wn{\mathbf{w}}

\newcommand{\tend}{\longrightarrow}
\newcommand\pto{\overset{\mathrm{p}}{\tend}}

\newcommand\eqd{\overset{\mathrm{d}}{=}}

\newcommand\NN{\mathbb{N}}

\newcommand\set[1]{\ensuremath{\{#1\}}}

\newcommand\bigpar[1]{\bigl(#1\bigr)}
\newcommand\Bigpar[1]{\Bigl(#1\Bigr)}
\newcommand\biggpar[1]{\biggl(#1\biggr)}
\newcommand\lrpar[1]{\left(#1\right)}
\newcommand\bigsqpar[1]{\bigl[#1\bigr]}
\newcommand\Bigsqpar[1]{\Bigl[#1\Bigr]}
\newcommand\biggsqpar[1]{\biggl[#1\biggr]}

\newcommand\op{o_{\mathrm p}}
\newcommand\Op{O_{\mathrm p}}

\newcommand\floor[1]{\lfloor #1 \rfloor}

\renewcommand\le{\leqslant}
\renewcommand\ge{\geqslant}

\newcommand{\indic}[1]{\mathbbm{1}_{{\{{#1}\}}}}

\renewcommand{\emptyset}{\varnothing} 

\DeclareMathOperator*{\NB}{\mathrm{NBin}}
\DeclareMathOperator*\Po{\mathrm{Po}}

\newcommand\mc{m_{\mathrm{c}}}
\newcommand\tc{t_{\mathrm{c}}}

\newcommand\marginal[1]{}

\renewcommand{\epsilon}{\varepsilon}
\newcommand{\eps}{\varepsilon}

\long\def\symbolfootnote[#1]#2{\begingroup
\def\thefootnote{\fnsymbol{footnote}}\footnote[#1]{#2}\endgroup}

\newenvironment{romenumerate}[1][-5pt]{
\addtolength{\leftmargini}{#1}\begin{enumerate}
 }{\end{enumerate}}

\newcounter{thmenumerate}

\newcommand{\dx}{\mathrm d}

\newcommand\bp{{\mathfrak{X}}}
\newcounter{case}
\newcommand\pfcase[1]{\refstepcounter{case}\smallskip\noindent\emph{Case
	\arabic{case}: #1} \noindent}

\newcommand\REM[1]{{\raggedright\texttt{[#1]}\par\marginal{XXX}}}

\newcommand\Bin{\mathrm{Bin}}
\newcommand\punkt{\xperiod}    
    
\newcommand\ie{i.e\punkt}
\newcommand\eg{e.g\punkt}

\newcommand{\as}{a.s\punkt}

\newcommand\whp{whp} 

\newcommand\ga{\alpha}

\newcommand\gd{\delta}

\newcommand\go{\omega}

\newcommand\ER{Erd\H{o}s--R{\'e}nyi}
\newcommand\ganx[1]{G^\ga_{n,#1}}

\newcommand\ganm{\ganx{m}}
\newcommand\gganx[1]{G^{\ga,*}_{n,#1}}

\newcommand\gganm{\gganx{m}}
\renewcommand\P{\Pr}
\newcommand\ntoo{\ensuremath{{n\to\infty}}}
 \newcommand\Gxdn{G^*_\dn}
\newcommand\qw{^{-1}}

\newcommand\qq{^{1/2}}
\newcommand\qqw{^{-1/2}}
\newcommand\qqq{^{1/3}}

\newcommand\qqqw{^{-1/3}}

\newcommand\mux{\mu}
\newcommand\bbR{\mathbb R}
 \newcommand\poo{p_\infty}
\newcommand\epsoo{\eps_\infty}
\newcommand\pn{p_n}
\newcommand\dnx{(d_v)_{v\in[n]}}
\newcommand\muu{\zeta}
\newcommand\fall[2]{\langle#1\rangle_{#2}}

\newcommand\bigabs[1]{\bigl\lvert#1\bigr\rvert}

\newcommand\mstop{m_\dagger}

\let\OLDthebibliography\thebibliography
\renewcommand\thebibliography[1]{
  \OLDthebibliography{#1}
  \setlength{\parskip}{0pt}
  \setlength{\itemsep}{0pt plus 0.3ex}
}

\usepackage[usenames,dvipsnames]{color}

\begin{document}

\title{Preferential attachment without vertex growth: \\ emergence of the giant component} 
\author{Svante Janson%
\thanks{Department of Mathematics, Uppsala University, 
PO Box 480, SE-751~06 Uppsala, Sweden. 
E-mail: {\tt svante.janson@math.uu.se}. 
Research partially supported by a grant from 
the Knut and Alice Wallenberg Foundation
and a grant from
the Simons foundation.} \ 
 and Lutz Warnke\thanks{School of Mathematics, Georgia Institute of Technology, Atlanta GA~30332, USA.
E-mail: {\tt warnke@math.gatech.edu}. Research partially supported by NSF Grant DMS-1703516 and a Sloan Research Fellowship.}}

\date{April 25, 2019}

\maketitle

\begin{abstract}
We study the following 
preferential attachment variant of the classical
Erd{\H{o}}s--R{\'e}nyi random graph process. 
Starting with an empty graph on~$n$~vertices, new edges are added one-by-one, 
and each time
an edge is chosen with probability roughly proportional to the product of
the current degrees of its endpoints (note that the vertex set is fixed).
We determine the asymptotic size of the giant component in the 
supercritical phase, confirming a conjecture of Pittel from~2010.
Our proof uses a simple method: 
we condition on the vertex degrees (of a multigraph variant), and 
use known results for the configuration~model.
\end{abstract}

\section{Introduction}
During the last two decades 
`dynamic' network models (which evolve/grow step-by-step)
have been of great interest in various different research areas, 
including combinatorics, 
probability theory, 
statistical physics, 
and network science, 
see \eg{}
\cite{BBRG,JLR,FKRG,
  Durrett2010,Hofstad2017,
  DM2003,KRBN2010,
  Barabasi2016}.
Part of the motivation stems from the fact that many real-world networks
(such as Facebook) also grow over time.
Widely studied models  include variants of the classical
Erd{\H{o}}s--R{\'e}nyi random graph
process~\cite{ER1960,BBRG,RWEP,RWAP,BR2013}  
and the modern
`scale-free' preferential attachment model made popular by
Barab\'asi and Albert \cite{BA1999,BRST2001,Barabasi2016},
which have strikingly different~features.

In this paper we consider a hybrid between between the 
Erd{\H{o}}s--R{\'e}nyi and Barab\'asi--Albert network models,
where the vertex set is fixed (as in the Erd\H{o}s--R{\'e}nyi case)
and edges are added with (one version~of) preferential attachment.
More precisely, let~$\alpha \in (0,\infty)$ be a parameter,
and write~$(G^{\alpha}_{n,m})_{m \ge 0}$ 
for the  random graph process
with fixed vertex set~$[n]=\set{1,\dots,n}$ 
where new edges are added one-by-one (starting with no edges) such that
the next edge connects two currently non-adjacent vertices~$v$ and~$w$ with probability proportional\footnote{Here we have tacitly normalized in a convenient way: namely, if the probability of adding the new edge~$\{v,w\}$ is proportional to~$(\chi d_v+\beta)(\chi d_w+\beta)$ with~$\beta,\chi>0$, then it is also proportional to~$(d_v+\alpha)(d_w+\alpha)$ with~$\alpha := \beta/\chi>0$.} to~$(d_v+\alpha) (d_w+\alpha)$, 
where~$d_v$ denotes the current degree of~$v$.
In intuitive words, edges are thus added according to a `rich-get-richer' 
preferential attachment mechanism (since vertices with higher degree are
more likely to be joined). 
Note that~$G^\ga_{n,m}$ has~$m$~edges.  
Furthermore, in the limit~$\ga\to\infty$ all edges are added with the same probability, 
so we recover the \ER{} random graph process.

The dynamic network model~$(G^{\alpha}_{n,m})_{m \ge 0}$ is so natural that is has been suggested
and studied multiple times (sometimes independently)
in the complex networks and combinatorial probability literature.
It~was first studied in~2010 by Pittel~\cite{Pittel}, who described it is a special
case of a more general model that he attributed to a suggestion 
by Lov{\'a}sz in~2002,  
where the next edge joins~$v$ and~$w$ with probability proportional to~$f(d_v)f(d_w)$ for some function~$f$; 
this in turn can be traced back to a suggestion\footnote{Erd{\H{o}}s and R{\'e}nyi proposed, even
  more generally, to study network models where the probability of joining~$v$ and~$w$ depends on the current degrees~$d_v$ and~$d_w$, see~\cite[p.~344]{ER1961}.}
of Erd{\H{o}}s and R{\'e}nyi~\cite{ER1961} from~1961 
motivated by more realistic modeling
(see also~\refS{sec:final:f}). 
From 2011 onwards~\citet{BCLSV2011}, \citet{RS2012} and~\cite[Example~7.9]{Jan2018} 
also studied a natural multigraph variant of~$\ganm$ (see \refS{sec:approx}), 
motivated by the emerging (multi)graph limit theory~paradigm. 
Furthermore, in~2012,
Ben-Naim and Krapivsky~\cite{BNK} proposed and studied
the model~$\ganm$ by statistical physics methods, motivated 
by the way connections are formed on Facebook (see also~\cite{Samalam2012} and~\refApp{sec:BNK:pw}).

In this paper we study the emergence of the giant component in this intriguing model, 
which is one of the most important and fascinating phase transitions in
random graph theory.
Pittel~\cite{Pittel} answered the basic question of existence and location of 
this phase transition in~$\ganm$: for any fixed~$\ga>0$   
he showed that at around~$m\approx \mc$ many steps 
the largest component typically changes from size~$\Theta(\log n)$ to size~$\Theta(n)$, 
where\footnote{The phase transition location~$\mc$ from~\eqref{def:mc} can easily be guessed via modern heuristics, see~\refApp{sec:mc:heuristic}.} 
\begin{equation}\label{def:mc}
  \mc := \frac{n}{2(1+\alpha^{-1})}
  = \frac{n\ga}{2(\alpha+1)}.
\end{equation}
A variant of this result for~$\alpha=1$ was also reported by Ben-Naim and Krapivsky~\cite{BNK}.
In fact, Pittel~\cite{Pittel} proved much stronger estimates on the 
size~$L_1(m)=L_1(G^{\alpha}_{n,m})$ of the largest component of~$G^{\alpha}_{n,m}$, 
in particular near the critical point~$\mc$.   
Focusing for simplicity on the `supercritical' phase 
(where the unique `giant' component has emerged as the largest component), 
his result~\cite[Theorem~1]{Pittel} can be written as follows.  
\unskip\footnote{See \refApp{sec:Pittel:pw} for Pittel's formulation of his supercritical giant component result.} 
If~$\eps=O(1)$ and~$\eps^4 n \to \infty$ as~$n \to \infty$, 
then, for a certain function
$\rho_\alpha(\eps)$ 
with $\rho_\alpha(\eps)=\Theta(\eps)$ as~$\eps \searrow 0$, we have\footnote{As usual, $o_p(1)$ denotes a quantity that converges to~$0$ in~probability as $n \to \infty$; see \eg{}~\cite{JLR,SJN6}.} 
\begin{align}\label{eq:Pittel:super:L1}
L_1\bigpar{\mc(1+\eps)} &= \rho_\alpha(\eps) n \cdot (1 +o_p(1)) .
\end{align}
Overall, Pittel's `finite-size scaling' results qualitatively recover several key features of the Erd{\H{o}}s--R{\'e}nyi phase transition~\cite{BR2009,BR2013}, 
in particular the fundamental `linear growth' of the $\Theta(\eps n)$--sized largest component, see~\eqref{eq:Pittel:super:L1}.  
However, for technical reasons his proof requires the extra assumption~$\eps^4 n \to \infty$, 
while it is natural to guess, and was conjectured by~\citet[pp.~621,649]{Pittel},
that the estimate~\eqref{eq:Pittel:super:L1} remains valid under the weaker 
supercritical condition~$\eps^3n \to \infty$ 
known from the Erd{\H{o}}s--R{\'e}nyi reference model.

The main purpose of the present paper is to extend Pittel's result and
verify his nice conjecture
(which also appears in 
the recent book by Frieze and Karo{\'n}ski~\cite[Section~17.5]{FKRG}).
We further extend the result by allowing~$\ga$ to depend on~$n$. 
Moreover, and at least as important, we do this using a simpler method
than the one in~\cite{Pittel}: we use known results for the
configuration model to derive the results rather quickly.

\subsection{Main results}\label{sec:main}
Our first result determines the asymptotic size of the giant component in
the entire supercritical phase,
thus confirming Pittel's nearly 10-year-old conjecture (see also \refApp{sec:Pittel:pw}).  
Furthermore, \eqref{eq:rho:linear}~below identifies the precise linear
growth-rate of the giant, generalizing and rigorizing a
statistical physics result by Ben-Naim and Krapivsky~\cite{BNK} 
from~\citeyear{BNK} for the special case~$\alpha=1$ (see~\refApp{sec:BNK:pw}). 
\begin{theorem}[Extending~\citet{Pittel}]\label{thm:main1}%
Fix~$\alpha \in (0,\infty)$. 
If\/~$\eps=\eps(n)=O(1)$ and~$\eps^3n \to \infty$ as $n \to \infty$, then
\begin{align}
\label{eq:thm:main1:giant:L1}
L_1\bigpar{\mc(1+\eps)} &=  \rho_\alpha(\eps) n \cdot (1+\op(1)) ,
\end{align}
where the continuous function~$\rho_\alpha : (0,\infty) \to (0,1]$
is given by \eqref{emma} 
together with \eqref{sofie}; 
it satisfies~$0<\rho_\alpha(\eps) < 2\eps$ and 
\begin{equation}\label{eq:rho:linear}
  \rho_\alpha(\eps)  = \frac{2\eps}{1+2/\alpha} + O(\eps^2),
  \qquad\text{as }\eps \searrow 0.
\end{equation}
\end{theorem}
\noindent
From the perspective of mathematical physics, this result places 
the preferential attachment process into the same universality class as the Erd{\H{o}}s--R{\'e}nyi reference model 
(with largest `supercritical' component of order~$\eps n$ and largest `subcritical' component of order $\eps^{-2}\log(\eps^3n)$, 
both under the condition~$\eps^3n \to \infty$, cf.~\cite{Pittel}). 
This kind of universality is only known for relatively few network models, 
including random regular graphs~\cite{NP2010}, the configuration model~\cite{OR2012}, 
hypercube percolation~\cite{vdHN2012,HN2016}, and bounded-size Achlioptas processes~\cite{RWapbsr}.

Since the limiting case~$\alpha \to \infty$ of the preferential attachment process
recovers the uniform Erd{\H{o}}s--R{\'e}nyi process~$(G^{\mathrm{ER}}_{n,m})_{m \ge 0}$,
it is natural to wonder under what conditions this 
Erd{\H{o}}s--R{\'e}nyi approximation holds rigorously in the case 
when~$\ga=\ga(n)$ is finite but tends to infinity as~$n \to \infty$.
Our main giant component result, which is the following extension of \refT{thm:main1}, 
allows us to answer this intriguing question
by allowing for~$\alpha=\alpha(n)$, including~$\ga(n)\to\infty$.
To establish uniqueness of the `giant component' (as in~\cite{Pittel}), 
\refT{thm:main} also includes a weak estimate on the 
size~$L_2(m)=L_2(G^{\alpha}_{n,m})$ of the second largest component of~$G^{\alpha}_{n,m}$.  
We henceforth use the convention~that $x/\infty =0$ for any finite~$x$.
\begin{theorem}[Main giant component result]\label{thm:main}
Assume that~$\alpha=\alpha(n) \to a \in (0,\infty]$ as $n \to \infty$. 
If~$\eps=\eps(n)=O(1)$ and~$\eps^3 n \to \infty$ as~$n \to \infty$, then 
\begin{align} %
\label{eq:thm:main:giant:L1}
L_1\bigpar{\mc(1+\eps)} &= \rho_a(\eps) n \cdot (1 +o_p(1)), \\ 
\label{eq:thm:main:giant:L2}
L_2\bigpar{\mc(1+\eps)}  & = o_p(1) \cdot L_1\bigpar{\mc(1+\eps)}  ,
\end{align}
where the function~$\rho_a:(0,\infty) \to (0,1]$ is as in \refT{thm:main1},
and~$\rho_\infty(\eps)$ satisfies 
\begin{align}
  \label{rhooo}
1-\rho_\infty(\eps)=e^{-(1+\eps)\rho_\infty(\eps)}. 
\end{align}
In particular, \eqref{eq:rho:linear} holds for any $\ga\in(0,\infty]$.
\end{theorem}
\noindent
\refR{rem:NB:rho} also shows that~$\lim_{\alpha \to \infty}\rho_\alpha(\eps)=\rho_\infty(\eps)$. 
Recognizing~\eqref{rhooo} as a standard branching process equation, 
the largest component of the Erd{\H{o}}s--R{\'e}nyi process~$(G^{\mathrm{ER}}_{n,m})_{m \ge 0}$
is well-known~\cite{BR2009,BR2013} to satisfy
\begin{equation}\label{eq:ER:L1}
L_1\Bigpar{G^{\mathrm{ER}}_{n,\tfrac{n}{2}(1+\eps)}}=\rho_\infty(\eps)n \cdot (1+o_p(1)) 
\end{equation}
when~$\eps=O(1)$ and~$\eps^3n \to \infty$.  
From \refT{thm:main} it is easy to deduce (using continuity of~$\rho_\infty$) 
that the preferential attachment process has the same giant component~behaviour 
when~$\alpha(n) \to \infty$ sufficiently~fast. 
\begin{corollary}[Supercritical Erd{\H{o}}s--R{\'e}nyi behaviour]\label{cor:main}
Assume that~$\alpha=\alpha(n) \to \infty$ as~$n \to \infty$. 
If~$\eps=\eps(n)=O(1)$, $\eps^3 n \to \infty$ and~$\alpha \eps \to \infty$ as~$n \to \infty$, 
then 
\begin{equation}\label{eq:ER:alpha}
L_1\bigpar{\tfrac{n}{2}(1+\eps)} = 
L_1\Bigpar{G^{\alpha}_{n,\tfrac{n}{2}(1+\eps)}}
= \rho_\infty(\eps) n \cdot  (1+\op(1)) .
\end{equation}
In particular, $\eps=O(1)$ and~$\eps^3 n \to \infty$ imply~\eqref{eq:ER:alpha} when~$\alpha = \Omega(n^{1/3})$. 
\end{corollary}

For the interested reader we include two remarks about 
the function~$\rho_\ga$ and the condition~$\eps=O(1)$. 
\begin{remark}
As shown in the proof in \refS{sec:L1}, $\rho_\ga(\eps)$ is an analytic
function of $\ga\in(0,\infty)$  and $\eps\in(0,\infty)$; moreover,
it extends analytically to $\eps\in[0,\infty)$, as shown in \refR{Rfit}.
\end{remark}

\begin{remark}\label{Rnobound}
The condition~$\eps=O(1)$ in Theorems~\ref{thm:main1}--\ref{thm:main} 
can be removed.
This is a trivial consequence of the monotonicity of the process and the
fact that $\rho_\ga(\eps)\to1$ as $\eps\to\infty$, see \refR{rem:NB:rho}:
it suffices to consider the case $\eps\to\infty$, and then we may for any
$\eta>0$ choose $\eps_0$ such that $\rho_a(\eps_0)>1-\eta$ and then
\whp{} (i.e., with probability tending to one as~$n \to \infty$) we have
\begin{align}
n\ge L_1(\mc(1+\eps)) \ge L_1(\mc(1+\eps_0)) \ge (\rho_a(\eps_0)-\eta)n\ge
(1-2\eta)n\ge(\rho(\eps)-2\eta)n, 
\end{align}
which also implies that~$L_2\le 2\eta n$ \whp.
\end{remark}

\subsection{Comments}\label{sec:comments}
To motivate why the phase transition of~$G^{\alpha}_{n,m}$
has some Erd{\H{o}}s--R{\'e}nyi features (like linear growth of the giant component), 
let us point out that distinct tree-components~$C_1,C_2$ merge with probability proportional~to
\begin{equation}\label{eq:tree:rate}
\sum_{v \in C_1}(d_v+\alpha) \cdot \sum_{w \in C_2}(d_w+\alpha) = \bigsqpar{(2+\alpha)|C_1|-2} \cdot \bigsqpar{(2+\alpha)|C_2|-2}.
\end{equation}
In concrete words, trees thus merge with rates that are approximately proportional 
to the rates~$|C_1| \cdot |C_2|$ from the Erd{\H{o}}s--R{\'e}nyi process. 
Since the phase transition is usually dominated by the contribution of tree-like components 
(e.g., by counting the vertices in `small' trees outside of the giant component), 
after some hand-waving it thus becomes plausible to observe some key features of the Erd{\H{o}}s--R{\'e}nyi reference model. 

Our actual proof takes a surprisingly simple different route (instead of trying to leverage the above tree-heuristic).
Indeed, we first show that a natural multigraph variant of~$G^{\alpha}_{n,m}$ has, conditioned on its degree sequence, 
the same distribution as the well-known configuration model for random multigraphs. 
By noting that in the multigraph variant the degrees evolve nearly independently (when looked at in the right~way), 
we can also get very strong control over the resulting asymptotic degree sequence 
of negative binomial form
(which contrasts not only the Poisson distribution of the \ER{}~model, 
but also the power-law distributions observed in many preferential attachment~models).
These two results together allow us to study the multigraph variant of~$G^{\alpha}_{n,m}$ via standard results for the configuration model, 
which then easily gives the asymptotic size of the giant component in~$G^{\alpha}_{n,m}$.
See \refS{sec:overview} for a detailed proof overview.

We mention that our arguments are quite different from Pittel~\cite{Pittel},
who  studies the multigraph variant of~$G^{\alpha}_{n,m}$ via involved
enumerative techniques.  
In fact, he notes~\cite[p.~643]{Pittel} that 
by conditioning on the degree sequence,
it might be possible to use known results for the configuration model;  
however, in the paper he used a different approach, 
partly because it seemed difficult to  
verify the required degree conditions. 
(We will see that this is not so difficult, using a continuous-time
construction. Moreover, we have the advantage of being able to use a
stream-lined version~\cite{JL2009} of the 
original phase transition~result~\cite{MR1998} for the configuration~model.)

Finally, while the focus of the present paper is on the giant component,
it is important to note that our proof method 
can also be used to study other properties of~$\ganm$ such as the $k$-core, see~\refS{sec:other}.

\subsection{Organization}
In Section~\ref{sec:overview} we give a detailed overview of our proof strategy. 
In particular, we state several technical auxiliary results, 
which we later prove in \refS{sec:approx:proofs}. 
Furthermore, in \refS{sec:heuristic} we give a heuristic argument of our main giant component result, 
which we make rigorous in \refS{sec:L1}. 
In \refS{sec:final} we then discuss extensions, variants, other properties, and open problems. 
Finally, \refApp{sec:mc:heuristic} contains a simple heuristic for the phase transition location~$\mc$, 
and \refApp{sec:previous} shows how our results are compatible with previous~work.

\section{Proof structure}\label{sec:overview} 
In this section we outline our high-level proof strategy 
for Theorems~\ref{thm:main1}--\ref{thm:main}, which proceeds roughly as follows.  
After introducing a suitable random multigraph variant~$(G^{\alpha,*}_{n,m})_{m \ge 0}$ of~$(G^{\alpha}_{n,m})_{m \ge 0}$,
we shall derive three basic auxiliary results 
(whose proofs are deferred to \refS{sec:approx:proofs}). 
First, Theorem~\ref{thm:transfer} shows that results for the random multigraph~$G^{\alpha,*}_{n,m}$ 
transfer to the original  random graph~$G^{\alpha}_{n,m}$. 
Secondly, Theorem~\ref{thm:equivalence} shows that, by conditioning on its degree sequence, 
we can study~$G^{\alpha,*}_{n,m}$ via the widely studied configuration model~$\Gxdn$ for random multigraphs. 
Thirdly, Theorem~\ref{thm:degree} determines the typical degree sequence of~$G^{\alpha,*}_{n,m}$. 
The crux is that these three results  together allow us to study~$G^{\alpha}_{n,m}$ by applying 
standard results for random graphs with given degree sequences
(see Theorem~\ref{thm:giant:conf}), 
and in Section~\ref{sec:heuristic} we outline how this reduction makes our 
main giant component result plausible (the full details are deferred to Section~\ref{sec:L1}).

\subsection{A multigraph variant: reduction and auxiliary results}\label{sec:approx}
We start by introducing a convenient multigraph variant of~$(G^{\alpha}_{n,m})_{m \ge 0}$, which allows for loops and multiple edges
(this natural model was also used by \citet{Pittel}, and further studied
from different aspects in~\cite{BCLSV2011,RS2012,Jan2018}).  
We write~$(G^{\alpha,*}_{n,m})_{m \ge 0}$ for the
random multigraph process 
with fixed vertex set~$[n]$ and parameter~$\alpha\in (0,\infty)$, 
where edges are added one-by-one (starting with no edges) such that
the next edge connects distinct vertices~$v$ and~$w$ with probability
proportional to~$2(d_v+\alpha) (d_w+\alpha)$,
and forms a loop at vertex~$v$ with probability proportional to
$(d_v+\alpha)(d_v+1+\alpha)$;
here~$d_v$ denotes the current degree of~$v$ 
(as usual, each loop is counted as two edges at its~endpoint). 
In this paper we shall first prove our main results 
for the random multigraph~$G^{\alpha,*}_{n,m}$, 
which turns out to be much easier to analyze than~$G^{\alpha}_{n,m}$; 
we record this intermediate goal for later~reference. 
\begin{theorem}[Main multigraph result]\label{Tmulti}
Theorems~\ref{thm:main1}--\ref{thm:main}
also hold for the random multigraph~$G^{\alpha,*}_{n,m}$. 
\end{theorem}

\subsubsection{Approximating~$G^{\alpha}_{n,m}$ by
  the multigraph~$\gganm$}\label{sec:approx1} 
Our first auxiliary result allows us to study~$G^{\alpha}_{n,m}$ via the 
random multigraph~$G^{\alpha,*}_{n,m}$ 
(so that Theorems~\ref{thm:main1}--\ref{thm:main} eventually follow from \refT{Tmulti}).
In words, Theorem~\ref{thm:transfer} implies that 
\whp--results for $G^{\alpha,*}_{n,m}$ routinely transfer to~$G^{\alpha}_{n,m}$ 
when~$m=O(n)$ and~$\alpha = \Omega(1)$.
This lemma is basically the same as 
\cite[Corollary~3]{Pittel};
one difference is that we do not restrict to constant~$\alpha$.
(Another difference
is that we only give an $o(1)$ bound for the  additive term 
in~\eqref{eq:transfer}--\eqref{eq:transfer:proc}.
This can easily be improved, but we do not need this.)
\begin{theorem}
  [Transfer statement: from~$G^{\alpha,*}_{n,m}$ to~$G^{\alpha}_{n,m}$,
  partly~\cite{Pittel}]
  \label{thm:transfer}%
Given~$C,\alpha_0>0$, there is~$B=B(C,\alpha_0) > 0$ such that the 
following holds whenever~$1 \le m \le C n$ and $\alpha \ge \alpha_0$.   
For any set~$\cG_{n}$ of graphs with $m$~edges and 
vertex set~$[n]$, we have 
\begin{equation}\label{eq:transfer}
\Pr(G^{\alpha}_{n,m} \in \cG_{n}) \le B \cdot \Pr( G^{\alpha,*}_{n,m} \in \cG_{n})
+ o(1).
\end{equation}
\end{theorem}
\begin{remarkttt}\label{rem:transfer}%
The proof shows more generally that, for any set~$\cG_{n,m}$ of graph 
sequences~$(G_0,\ldots, G_m)$ with vertex set~$[n]$, we have 
\begin{equation}\label{eq:transfer:proc}
\Pr( (G^{\alpha}_{n,i})_{0 \le i \le m} \in \cG_{n,m}) \le B \cdot \Pr( (G^{\alpha,*}_{n,i})_{0 \le i \le m} \in \cG_{n,m} ) + o(1). 
\end{equation}
\end{remarkttt}
\noindent
The proof strategy is to compare the stepwise conditional probabilities of the added edges, where 
we may clearly restrict to simple graph sequences (without loops or multiple edges). 
By construction, the conditional probability of adding the new edge~$\{v,w\}$ is in both processes 
proportional to~$2(d_v+\alpha)(d_v+\alpha)$, but the normalizing factors in the denominator differ 
slightly (since only one of these processes allows for loops and multiple edges; 
see~\eqref{eq:Gi}--\eqref{eq:Hi} in~\refS{sec:transfer}). 
It turns out that during the first~$m=O(n)$ steps the corresponding normalizing factors of both 
processes are extremely close together, which eventually allows us to establish~\eqref{eq:transfer:proc}, 
from which inequality~\eqref{eq:transfer} follows immediately.  
See~\refS{sec:transfer} for the full details.

\subsubsection{Reducing the multigraph $\gganm$ to the configuration model}\label{sec:approx2} 
Our second auxiliary result allows us to study~$G^{\alpha,*}_{n,m}$ via the
well-known 
\emph{configuration model~$\Gxdn$} for random multigraphs,
see \eg{}~\cite{BBRG,FKRG,Hofstad2017};
it says
that~$G^{\alpha,*}_{n,m}$ conditioned 
on its degree sequence~$d(G^{\alpha,*}_{n,m})=\dn=\dnx$
has the same distribution as the
configuration model~$\Gxdn=([n],E_\dn)$, which we for concreteness here
define as follows:
Let $\cS_\dn$ be the 
set of all $2m$-element sequences in which each vertex~$v \in [n]$
appears~$d_v$ times,
 pick a uniform random vertex sequence~$\wn=(w_1, \ldots, w_{2m}) \in
 \cS_\dn$, and define  
the edge multiset~$E_\dn:=\set{w_1w_2, \ldots, w_{2m-1}w_{2m}}$.
\unskip
\footnote{This construction indeed gives the usual configuration model,
  since~$E_\dn$ has the same distribution as the edges of a uniform random
  matching of the $2m$-element multiset in which each vertex~$v \in [n]$
  appears~$d_v$~times.} 
As said above, this property was noted (but not exploited) by
\citet{Pittel} (at least after conditioning the multigraph on being simple);
it has also been noted or used in~\cite{RS2012,Jan2018}.
\begin{theorem}
  [Conditional equivalence: $G^{\alpha,*}_{n,m}$ and~$\Gxdn$,~\cite{Pittel,RS2012,Jan2018}]    
  \label{thm:equivalence}%
For any~$m \ge 1$ and any degree sequence~$\dn=(d_v)_{v \in [n]}$ with~$\sum_{v \in [n]}d_v=2m$, 
the random multigraph~$G^{\alpha,*}_{n,m}$ conditioned on having degree sequence~$\dn$ 
has the same distribution as the configuration model~$\Gxdn$. 
In other words,
for any set~$\cG_{n}$ of multigraphs with $m$ edges and 
vertex set $[n]$, we~have 
\begin{equation}\label{eq:equivalence}
  \Pr\bigpar{G^{\alpha,*}_{n,m} \in \cG_{n} \mid d(G^{\alpha,*}_{n,m}) = \dn}
  = \Pr\bigpar{\Gxdn \in \cG_{n}} .
\end{equation}
\end{theorem}
\noindent
As we shall see in~\refS{sec:equivalence}, the random multigraph process was carefully designed 
to make~\eqref{eq:equivalence} true `by construction'. To make this plausible, recall that the 
conditional probability of adding the non-loop edge~$\{v,w\}$ in the next step is proportional 
to~$2 \cdot (d_v+\alpha) \cdot (d_w+\alpha)$. 
This `factorization'  suggests that we can alternatively construct the~$m$ edges 
of~$G^{\alpha,*}_{n,m}$ as follows: we first generate the vertex sequence~$w_1, \ldots, w_{2m}$ 
(each time vertex~$v$ is  chosen with probability proportional to~$d_{v}+\alpha$) 
and then join them pairwise to the~$m$ edges~$w_1w_2, \ldots, w_{2m-1}w_{2m}$. 
In~\refS{sec:equivalence} we show that (a version of) this construction indeed gives 
the correct distribution (this is the place where the special treatment of loops is crucial). 
It furthermore turns out that every vertex sequence~$(w_1, \ldots, w_{2m}) \in \cS_\dn$ arises 
with the same probability (see~\eqref{eq:excha} in \refS{sec:equivalence}), 
which by the described construction of~$\Gxdn=([n],E_\dn)$ then easily gives the desired 
conditional equivalence. See~\refS{sec:equivalence} for the full details.

\subsubsection{Approximating the degree sequence of $\gganm$}
\label{sec:approx3} 
Our third auxiliary result states that the degree sequence
of~$G^{\alpha,*}_{n,m}$ is 
asymptotically a \emph{negative binomial distribution}~$Y \sim \NB(\alpha,p)$ 
with shape parameter~$\alpha$ and suitable probability~$p=p(\alpha,m/n)$, i.e., 
\begin{equation}\label{eq:NB}
\Pr(Y=r):= 
\binom{\alpha+r-1}{r}(1-p)^{\alpha} p^r 
= \frac{\prod_{0 \le j < r}(\alpha+j)}{r!} (1-p)^{\alpha} p^r 
\qquad \text{for \ $r \in \NN=\{0,1,\ldots\}$.}
\end{equation}
For notational convenience, 
given a degree sequence~$\dn=(d_v)_{v \in [n]}$ 
and an integer~$k\ge0$, we write 
\begin{align}
\label{dxp:mux}
  \pi_k(\dn):= \frac{1}{n}\sum_{v \in [n]} \indic{d_v=k} 
\quad \text{ and } \quad 
  \mux_k(\dn):= \frac{1}{n}\sum_{v\in[n]} d_v^k 
\end{align}
for the proportion of vertices with degree~$k$,  
and the~$k\:$th moment of the degree of a random vertex, respectively. 
We now state a limit theorem for the degrees (aiming at simplicity rather than the
widest generality). 
\begin{theorem}[Degree sequence of~$G^{\alpha,*}_{n,m}$: negative
  binomial]\label{thm:degree}
  Suppose that $\ntoo$, $m=\Theta(n)$, 
  and $\ga=\ga(n)=\Omega(1)$.
  Let\/ $\dn$ be the (random) degree sequence of $\gganm$, and
  let\/ $Y=Y_n\sim \NB(\ga,p_n)$ with
  \begin{align}
    \label{pn}
    \pn:=2m/(n\ga+2m).
      \end{align}
      Then, for any sequence $\go(n)\to\infty$,
      for every integer~$k\ge0$,
  \begin{align}
	\label{l3a}
    \pi_k(\dn)& = \P(Y_n=k)+\op\bigpar{\go(n)n\qqw}, \\
		\label{l3b}
    \mux_k(\dn)& = \E Y_n^k +\op\bigpar{\go(n)n\qqw},
    \end{align}
    where~$\E Y_n^k=O(1)$.
\end{theorem}
\begin{remarkttt}
  The estimates \eqref{l3a}--\eqref{l3b} with an arbitrary $\go(n)\to\infty$
can equivalently be written as
$\pi_k(\dn) = \P(Y_n=k)+\Op\bigpar{n\qqw}$
and
$\mux_k(\dn) = \E Y_n^k +\Op\bigpar{n\qqw}$,
see \eg{}  \cite[Lemma~3]{SJN6}.
\end{remarkttt}
\begin{remarkttt}\label{rem:degree}%
  We have
  $\E Y_n=2m/n=\sum_{v \in [n]} d_v/n$, $\E Y_n(Y_n-1) = (2m/n)^2 (1+\alpha^{-1})$, and $\E Y_n(Y_n-1)(Y_n-2) = (2m/n)^3 (1+\alpha^{-1})(1+2\alpha^{-1})$; see \refL{cl:NB}. 
\end{remarkttt}
\noindent
The proof strategy hinges on the fact that, in~$G^{\alpha,*}_{n,m}$,  
the degree of a vertex equals the number of its concurrences 
in the auxiliary sequence~$w_1, \ldots, w_{2m}$ introduced above. 
To study these statistics we then switch to continuous time, 
and identify the degree of each vertex~$v \in [n]$ 
with a suitable independent birth process
(with initial value~$0$ and birth rates~$\lambda_k:=k+\alpha$). 
 In \refS{sec:degree}, we show that if we sequentially record 
the vertices which give birth, then the resulting vertex sequence~$w_1, \ldots, w_{2m}$ indeed has the correct distribution. 
This continuous-time embedding ensures 
(i)~that the birth processes and thus the degrees evolve independently, 
and (ii)~that each birth process has an explicit distribution at time~$t$, 
which turns out to be of negative binomial form (see~\eqref{eq:Dvt} in \refS{sec:degree}). 
These two
properties make it easy to approximate the 
degree sequence up to the desired precision. 
See \refS{sec:degree} for the full~details.

\subsection{Reduction to configuration model~$\Gxdn$: giant component heuristics}\label{sec:heuristic} 
The punchline of the auxiliary results above
is that we can obtain results for~$G^{\alpha}_{n,m}$ 
by applying standard results for the well-understood configuration
model~$\Gxdn$, where $\dn$ is random but approximates a negative binomial
distribution.
Armed with this reduction to the configuration model~$\Gxdn$, 
the plan is to then estimate the size of the largest component 
by applying the following result of Janson and Luczak
\cite[a special case of Theorems~2.3--2.4]{JL2009}, 
which is a convenient extension of the pioneering result by
Molloy and Reed~\cite{MR1998}. 
(The moment condition~$\mux_5(\dn)=O(1)$ 
can be weakened to lower moments, see~\cite{JL2009,SJ313}, but we
do not need~this.)
\begin{theorem}[Phase transition in~$\Gxdn$,~\cite{JL2009}]\label{thm:giant:conf}%
Suppose that, for each~$n \ge n_0$, $\dn=(d_v)_{v \in [n]}$ is a 
sequence of non-negative integers with~$\mux_5(\dn)=O(1)$ 
such that~$\sum_{v \in [n]} d_v$ is even. 
Furthermore, suppose that~$D \in \NN$ is a random variable that is 
independent of~$n$ such that~$\E D \in (0,\infty)$, $\Pr(D=1)>0$, 
and~$\pi_k(\dn) \to\Pr(D=k)$ 
as~$n \to \infty$, for every~$k \ge 0$.
Then, writing~$L_1=L_1(\Gxdn)$ and~$L_2=L_2(\Gxdn)$ for the sizes of 
the largest and second largest component of~$\Gxdn$, the following~holds. 
\begin{romenumerate}
\parskip 0em  \partopsep=0pt \parsep 0em 
\item\label{thm:supercr}
If~$\E D(D-2) >0$, then there is a unique~$\xi \in (0,1)$ satisfying~$\E D \xi^D = \xi^2 \E D$, and, furthermore,   
\begin{equation}\label{eq:supercr}
L_1/n \pto 1-\E \xi^D > 0 \quad \text{ and } \quad L_2/n \pto 0 .
\end{equation}
\item\label{thm:critical}
If~$\E D(D-2) =0$ and~$\muu_n := \sum_{v \in [n]}d_v(d_v-2)$ satisfies~$\muu_n > 0$ and~$n^{-2/3} \muu_n \to \infty$ as~$n \to \infty$, then 
\begin{equation}\label{eq:critical}
L_1 = \Bigpar{\frac{2 \E D}{\E D(D-1)(D-2)}+o_p(1)} \muu_n  \quad \text{ and } \quad L_2 = o_p(\muu_n) .
\end{equation}
\item\label{thm:subcr}
If~$\E D(D-2) \le 0$, then~$L_1/n \pto 0$.  
\end{romenumerate}
\end{theorem}
%

In the following heuristic discussion we shall make our giant component 
results for $G^{\alpha}_{n,m}$ plausible (with a focus on \refT{thm:main1}). 
To this end, as discussed above, it suffices to study the giant component of~$\Gxdn$ 
for `typical' degree sequences~$\dn$, where
we shall below (for simplicity) assume that we can 
apply \refT{thm:giant:conf} with the random variable~$D$ being approximately 
equal to~$Y=Y_n \sim \NB(\alpha,p)$ from Theorem~\ref{thm:degree}. 
Regarding the point \emph{when the giant component emerges}, 
in the usual informal language \refT{thm:giant:conf} 
(together with Remark~\ref{rem:degree}) 
states that there is a giant component if and only if 
\begin{equation}\label{eq:mc:location}
\E D (D-2) \approx \E Y(Y-2) =\E Y (Y-1) - \E Y =  \frac{2m}{n}\biggpar{\frac{2m}{n}\Bigpar{1+\alpha^{-1}}-1}
\end{equation}
is larger than zero, which makes the phase transition location~$\mc \approx n/[2(1+\alpha^{-1})]$ 
plausible (see \refApp{sec:mc:heuristic} for an alternative heuristic). 
%
Regarding the \emph{size of the largest component} for~$m=\mc(1+\eps)$ with~$\eps=\eps(n) \to 0$, 
in view of~\eqref{eq:mc:location} we have~$\E D (D-2) \approx \E Y(Y-2) = \Theta(\eps) \to 0$. 
In this case \refT{thm:giant:conf}\ref{thm:critical} intuitively predicts 
(together with~\eqref{l3b} and \refR{rem:degree}) 
that the size of the largest component is approximately  
\begin{equation*}
\frac{2 \E D}{\E D (D-1) (D-2)} \cdot \sum_{v \in [n]} d_v(d_v-2) \approx \frac{2 \E Y\, \E Y(Y-2) \, n}{\E Y (Y-1)(Y-2)} 
\approx \frac{2\eps n}{1+2\alpha^{-1}} ,
\end{equation*}
which makes~$L_1\bigpar{\mc(1+\eps)} \approx \rho_\alpha(\eps) n$ with~$\rho_\alpha(\eps) \approx 2\eps/(1+2\alpha^{-1})$ as~$\eps \searrow 0$ plausible.  
In fact, our above application of \refT{thm:giant:conf}\ref{thm:critical} tacitly required that the parameter
\[
n^{-2/3}\cdot \sum_{v \in [n]} d_v(d_v-2) \approx n^{1/3} \cdot \E Y(Y-2) = \Theta(n^{1/3}\eps)
\]
tends to infinity, which makes the \emph{assumption~$\eps^3n \to \infty$} and thus \refT{thm:main1} plausible. 
See \refS{sec:L1} for a rigorous version of the above heuristic arguments (in the more general setting of \refT{thm:main}).

\section{Proofs of auxiliary results}\label{sec:approx:proofs}
In this section we prove the three basic auxiliary results stated in 
Sections~\ref{sec:approx1}--\ref{sec:approx3}.
As noted above, the results have partly been shown earlier, but for
completeness we give complete proofs  of the versions used here.
We consider \refT{thm:transfer} last, since we find it
convenient to use \refT{thm:degree} in the proof.

\subsection{Proof of Theorem~\ref{thm:equivalence}: conditional equivalence of~$G^{\alpha,*}_{n,m}$ and~$\Gxdn$}\label{sec:equivalence}
The conditional equivalence result of Theorem~\ref{thm:equivalence} 
can be shown in several ways, including enumeration~\cite{Pittel} 
and exchangeability~\cite{RS2012,Jan2018} approaches. 
Inspired by P\'olya urn arguments, here we shall use an 
elementary approach that avoids cumbersome explicit calculations 
by defining appropriate random variables
(which also facilitates the upcoming degree sequence arguments). 
To this end we introduce a random~sequence 
\begin{equation}\label{def:Wj}
(W_j)_{j \ge 1}
\end{equation}
of vertices from~$[n]$, 
by defining~$W_{i+1}$ to have the conditional 
probability distribution with 
\begin{equation}\label{eq:def:Wi}
\Pr\bigpar{W_{i+1}=v \: \big| \: W_1, \ldots, W_i} = \frac{\sum_{j \in [i]}\indic{W_j=w}+\alpha}{i+\alpha n} \qquad \text{for all~$v \in [n]$}.
\end{equation}
(This can be interpreted as the sequence of draws from a P\'olya urn with
$n$ colours, with initially $\ga$ balls of each colour, see \eg{}
\cite{Markov1917,Polya1930,JohnsonKotz}.)
By joining these vertices pairwise to edges, for each~$i \ge 0$ we then obtain 
a multigraph with vertex set~$[n]$ and 
edge multiset~$E_{n,i} := \{W_1W_2, \ldots, W_{2i-1}W_{2i}\}$, where~$E_{n,0}=\emptyset$. 
The resulting multigraph sequence is easily seen to have the same distribution as~$(G^{\alpha,*}_{n,i})_{i \ge 0}$. 
Indeed, since~$\sum_{j \in [2i]}\indic{W_j=v}=d_v(i)$ equals the degree of vertex~$v$ after~$i$ steps, 
this follows from the simple observation that the stepwise conditional distributions of the added edges are the same 
(the conditional probability of adding~$\{v,w\}$ as the next edge is 
proportional to~$2 \cdot (d_v(i)+\alpha)(d_w(i)+\alpha)$ when~$v \neq w$, 
and proportional to~$(d_v(i)+\alpha)(d_v(i)+1+\alpha)$ when~$v = w$). 
Now the proof of the desired conditional equivalence result is straightforward, 
since the above `pairwise joining of~$2m$~vertices' construction of the edge multiset 
is similar to the construction of the configuration model~$\Gxdn$ described in~Section~\ref{sec:approx2}.
\begin{proof}[Proof of Theorem~\ref{thm:equivalence}]
Let~$\dn=(d_v)_{v \in [n]}$ be a degree sequence with~$\sum_{v \in [n]}d_v=2m$. 
Recalling that~$\cS_\dn$ denotes the set of all $2m$-element sequences~$\wn = (w_1, \ldots, w_{2m})$ 
in which each vertex~$v \in [n]$ appears~$d_v$ times, below we shall 
write~$E(\wn):=\{w_1w_2, \ldots, w_{2m-1}w_{2m}\}$ for the associated edge multiset. 
For any multigraph~$G$ with degree sequence~$d(G)=\dn$, 
by the above-discussed construction of~$G^{\alpha,*}_{n,m}$ it follows that 
\[
\Pr\bigpar{G^{\alpha,*}_{n,m}=G \; \big| \; d(G^{\alpha,*}_{n,m})=\dn} = \frac{\Pr(G^{\alpha,*}_{n,m}=G)}{\Pr(d(G^{\alpha,*}_{n,m})=\dn)} \\
 = \frac{\sum_{\wn \in \cS_{\dn}: E(\wn)=E(G)} \Pr\bigpar{(W_1, \ldots, W_{2m}) = \wn}}{\sum_{\wn \in \cS_{\dn}} \Pr\bigpar{(W_1, \ldots, W_{2m}) = \wn }} .
\]
By multiplying the conditional probabilities from~\eqref{eq:def:Wi}
(and rearranging the factors in the numerator),
it now is straightforward to see that the above probabilities  
\begin{equation}\label{eq:excha}
\Pr\bigpar{(W_1, \ldots, W_{2m}) = \wn} = \frac{\prod_{v \in [n]} \prod_{0 \le j < d_v}(j+\alpha)}{\prod_{0 \le i < 2m}(i+\alpha n)}
\end{equation}
are the same for all~$\wn \in \cS_{\dn}$ 
(as they depend on the degree sequence~$\dn=(d_v)_{v \in [n]}$ only). 
Since the edge multiset of~$\Gxdn$ is defined as~$E(\wn)$ for 
a uniform random~$\wn \in \cS_{\dn}$ (see Section~\ref{sec:approx2}),  
it readily follows that 
\begin{equation}
\begin{split}
\Pr\bigpar{G^{\alpha,*}_{n,m}=G \; \big| \; d(G^{\alpha,*}_{n,m})=\dn} 
= \sum_{\wn \in \cS_{\dn}}\frac{\indic{E(\wn)=E(G)}}{|\cS_{\dn}|} 
= \Pr(\Gxdn=G) ,
\end{split}
 \end{equation}
which implies the claimed conditional equivalence (since~$G$ with degree sequence~$d(G)=\dn$ was~arbitrary).  
\end{proof}

\subsection{Proof of Theorem~\ref{thm:degree}: negative binomial degree sequence of~$G^{\alpha,*}_{n,m}$}\label{sec:degree}
For the degree sequence result of Theorem~\ref{thm:degree} it 
will be convenient to consider a continuous time embedding of 
the~$(W_j)_{j \ge 1}$ based construction of~$(G^{\alpha,*}_{n,i})_{i \ge 0}$,  
since this will give us more independence. 
(This is a special case of a general embedding for P\'olya urns, see \eg{}
\cite[Section~9.2]{AN1972} with extension in
\cite[Remark~4.2]{SJ154} and
\cite[Remark~1.11]{SJ169}.)
To this end, let
\begin{equation}\label{def:Dvt}
\Bigpar{\Bigpar{D_v(t)}_{t \in [0,\infty)}}_{v \in [n]}
\end{equation}
be independent pure birth processes 
with initial value~$0$ and birth rates~$\lambda_k := k+\alpha$ 
(i.e., the transition rate from state~$k$ to state~$k+1$).
Identifying vertex~$v \in [n]$ with the birth process~$D_v(t)$, 
the order of the random birth-times~$(\tau_{j})_{j \ge 1}$ 
(also defining~$\tau_0:=0$ for convenience)  
thus naturally induces a sequence of random vertices~$(W_j)_{j \ge 1}$ 
(the ones which gave birth at the corresponding times). 
Justifying our slight abuse of notation, it is not difficult to check that this vertex sequence 
has the same distribution as the sequence~$(W_j)_{j \ge 1}$ defined in Section~\ref{sec:equivalence}. 
Indeed, since vertex~$v$ occurs~$\sum_{j \in
  [i]}\indic{W_j=v}=D_v(\tau_{i})$ many times in~$(W_1, \ldots, W_i)$,
this follows from the simple observation that the stepwise conditional distributions 
of the selected vertices are the same 
(the conditional probability of selecting~$v$ as the next vertex is~$(D_v(\tau_{i})+\alpha)/(i+\alpha n)$, 
since the total rate equals~$\sum_{w \in [n]}(D_w(\tau_{i})+\alpha)=i+\alpha n$). 
For later reference we also record the standard fact
\unskip\footnote{By standard textbook results (see, e.g.,
  \cite[Section~6.8]{GS2001}) the functions~$p_k(t):=\Pr(D_v(t)=k)$
  with~$p_k(t)=\indic{k=0}$ are the unique solutions of the forward
  equations~$p'_k(t)=\lambda_{k-1}p_{k-1}(t)\indic{k \ge 1}-\lambda_k
  p_k(t)$.  
For~$\lambda_k=k+\alpha$ the solution of these differential equations 
turns out to be~$p_k(t)=\Pr(Y=k)$ with~$Y \sim \NB(\alpha,1-e^{-t})$ 
as defined in~\eqref{eq:NB}. 
This can alternatively be deduced from~\cite[Exercise~6.8.6]{GS2001}, and is
also explicitly stated in~\cite[(3.15)]{T89}, for example.}
that~$D_v(t)$ has a negative binomial distribution 
with shape parameter~$\alpha$ and probability~$1-e^{-t}$, i.e., 
\begin{equation}\label{eq:Dvt}
D_v(t) \sim \NB(\alpha,1-e^{-t}).
\end{equation}
Exploiting independence of the birth-processes, 
now the proof of the desired degree sequence result is conceptually straightforward, 
since in our vertex-based construction of~$G^{\alpha,*}_{n,m}$ the degree of vertex~$v$~equals
\begin{equation}\label{eq:Dwm:deg}
d_v(m)=\sum_{j \in [2m]}\indic{W_j=v}=D_v(\tau_{2m}) ,
\end{equation}
where~$\tau_{2m}$ will be highly concentrated. 
Before giving the details, we record some basic
properties of negative binomial random variables 
(deferring their proofs, which are rather tangential to the main~argument
here). 
For~$x\in\bbR$ and integer~$k\ge 0$,
we denote the falling factorial by
$\fall{x}{k}:=\prod_{0\le j<k}(x-j)$, where~$\fall{x}{0}=1$.
\begin{lemma}\label{cl:NB}
For~$Y \sim \NB(\alpha,p)$ the following holds for all~$\alpha \in
(0,\infty)$ and~$p \in [0,1)$.  
\begin{romenumerate}
\parskip 0em  \partopsep=0pt \parsep 0em 
	\item\label{cl:NB:mom}%
	We have, for every integer~$k \ge 0$,
      \begin{align}
        \label{nb:mom:mom1}
        \E \fall{Y}{k} = (\E Y)^k \cdot \prod_{1 \le j < k}(1+j/\alpha) 
				\quad \text{ with } \quad
				\E Y = \frac{\ga p}{1-p} .
      \end{align}
\item \label{cl:NB:pgf}
The probability generating function
and its derivative are given by, for~$|x|<1/p$,
\begin{align}\label{nb:pgf:pgf'}
  \E x^Y = \Bigpar{\frac{1-p}{1-px}}^\ga 
\quad \text{ and } \quad 
  \E [Y x^{Y-1}] = \ga p \frac{(1-p)^\ga}{(1-px)^{\ga+1}}.
\end{align}
	\item\label{cl:NB:mean}%
 If~$p=1-e^{-t}$, then~$t=\log(1+x/\alpha)$
 implies
 $p=1-1/(1+x/\alpha)=x/(x+\ga)$
 and $\E Y = x$.
\end{romenumerate}
\end{lemma}

\begin{proof}[Proof of \refT{thm:degree}]
By the  construction of~$G^{\alpha,*}_{n,m}$  discussed above,
we have, see \eqref{eq:Dwm:deg},
\begin{equation}\label{eq:thm:degree:pr}
\dn=d(G^{\alpha,*}_{n,m}) = \bigpar{D_v(\tau_{2m})}_{v \in [n]}.
\end{equation}
With an eye on~\eqref{dxp:mux}, we now introduce the auxiliary variables
\begin{align}
\label{Nk:Mq}
N_k(t) := \sum_{v \in [n]}\indic{D_v(t)=k}
\quad \text{ and } \quad 
M_k(t) &:= \sum_{v \in [n]}D_v(t)^k . 
\end{align}
We henceforth write~$\go=\go(n)$, and may (as usual) assume 
that~$\sqrt{\go n}\le 2m$ holds. 
With the goal of approximating~$\tau_{2m}$ by a \emph{deterministic}
auxiliary time~$t_m$, define
\begin{align}
\label{eq:tm:tpmm}
t_m := \log\Bigpar{1+\frac{2m}{\alpha n}} 
\quad \text{ and } \quad 
t^\pm_m := \log\Bigpar{1+\frac{2m\pm\sqrt{\go n}}{\alpha n}} .
\end{align}
These times are chosen such that, by \eqref{eq:Dvt} and \refL{cl:NB}\ref{cl:NB:mean}, we have 
\begin{align}
  D_v(t_m)&\eqd Y_n \sim \NB(\ga,p_n), \label{eddax} \\
  \label{eddapm}
  \E D_v(t^\pm_m)&=\frac{2m\pm\sqrt{\go n}}{ n} = \E D_v(t_m) \pm \sqrt{\go/n}.
\end{align}
Furthermore, since we assume~$\sqrt{\go n}\le 2m$ and~$m=O(n)$,
all terms on the right-hand side of~\eqref{eddapm} are~$O(1)$, 
so~\eqref{nb:mom:mom1} and~$\ga=\Omega(1)$ yield, for every fixed integer~$q \ge 0$, 
\begin{align}\label{snorre}
  \max_{t \in \{t_m^-,t_m,t_m^+\}}\E\fall{ D_v(t)}{q} =O(1) .
\end{align}
For every~$k \ge 1$ it is well-known that we may write~$x^k$ as a linear combination
\begin{align}
  \label{stirling}
  x^k=\sum_{1 \le q \le k}\genfrac\{\}{0pt}{}{k}{q}\fall{x}q ,
\end{align}
where the coefficients are the so-called Stirling numbers of the second kind. 
Hence~\eqref{snorre} also holds with~$\fall{ D_v(t)}{q}$ replaced by~$D_v(t)^{q}$, so that~$\E Y_n^q=O(1)$ by~\eqref{eddax}.  
Using independence of the birth processes, we also infer 
\begin{align}\label{sos}
\Var M_k(t_m^\pm) =  \sum_{v \in [n]}\Var \bigsqpar{D_v(t_m^\pm)^k} \le \sum_{v \in [n]} \E D_v(t_m^\pm)^{2k} = O(n) . 
\end{align}

We now approximate~$\tau_{2m}$. To this end we consider~$M_1(t)=\sum_{v \in [n]} D_v(t)$, where~\eqref{eddapm} yields
\begin{align}\label{jan}
  \E M_1(t_m^\pm) =n\E D_v(t_m^\pm)= 2m\pm\sqrt{\go n}.
\end{align}
Using Chebyshev's inequality and~\eqref{sos}--\eqref{jan} it follows that, \whp, $M_1(t_m^-) <2m$ and~$M_1(t_m^+) >2m$. 
Since~$M_1(\tau_{2m})=2m$ by~\eqref{eq:thm:degree:pr}--\eqref{Nk:Mq}, using time-monotonicity of~$M_1(t)$ this implies that, \whp,
\begin{align}\label{eq:tau2m}
  t_m^- < \tau_{2m} < t_m^+.
\end{align}

Next we focus on~$\mux_k(\dn)=\sum_{v \in [n]}D_v(\tau_{2m})^k/n = M_k(\tau_{2m})/n$, see~\eqref{dxp:mux} and~\eqref{eq:thm:degree:pr}--\eqref{Nk:Mq}.  
By combining~\eqref{nb:mom:mom1} and~\eqref{eddapm} with~\eqref{snorre} it follows that, for every fixed integer~$q \ge 0$, 
\begin{align*}
    \E \fall{D_v(t_m^\pm)}{q}
  = \lrpar{\frac{2m\pm\sqrt{\go n}}{2m}}^q \E \fall{D_v(t_m)}{q}
  = \E \fall{D_v(t_m)}{q} +O\bigpar{\sqrt{\go/n}}
  .
\end{align*}
Using~\eqref{stirling} again, we infer that~$\E D_v(t_m^\pm)^k = \E D_v(t_m)^k +O(\sqrt{\go/n})$. 
Consequently, using Chebyshev's inequality and~\eqref{sos} again, we deduce that, \whp,
\begin{align}\label{eq:conc:Mkt}
  M_k(t_m^\pm) = \E M_k(t_m^\pm) \pm \sqrt{\go n} 
    =n \E D_v(t_m)^k +O\bigpar{\sqrt{\go n}},
\end{align}
Since~$\E D_v(t_m)^k=\E Y_n^{k}$ by~\eqref{eddax}, 
using time-monotonicity of~$M_k(t)$ and~\eqref{eq:tau2m} it now follows that 
\begin{align}\label{noa}
  M_k(\tau_{2m}) 
      =n \E Y_n^{k} +\op\bigpar{\go  n\qq} ,
\end{align}
which in view of~$\mux_k(\dn)=M_k(\tau_{2m})/n$ establishes~\eqref{l3b}.

Finally, we turn to~$\pi_k(\dn)=\sum_{v \in [n]}\indic{D_v(\tau_{2m})=k}/n=N_k(\tau_{2m})/n$, 
see~\eqref{dxp:mux} and~\eqref{eq:thm:degree:pr}--\eqref{Nk:Mq}.  
Here the key observation is that~\eqref{eq:tau2m}, \eqref{eq:conc:Mkt}
and~\eqref{jan} together imply that, \whp,
\begin{align}\label{ham}
    |N_k(\tau_{2m})-N_k(t_m)| \le |M_1(t^+_m)-M_1(t^-_m)|
  \le |\E M_1(t^+_m)-\E M_1(t^-_m)| + 2\sqrt{\omega n} = 4 \sqrt{\omega n}
.
\end{align}
Furthermore, $\E N_k(t_m)=n\E \indic{D_v(t_m)=k}=n\Pr(Y_n=k)$ by~\eqref{eddax}, 
and~$\Var N_k(t_m) =O(n)$ by independence of the birth processes. 
Hence, using~\eqref{ham} and Chebyshev's inequality it follows that, \whp,
\begin{equation}\label{eq:Nj:bound}
  \bigabs{N_k(\tau_{2m}) - n \Pr(Y_n=k)}
  \le \bigabs{N_k(t_m) - n \Pr(Y_n=k)} +4\sqrt{\go n} \le 5 \sqrt{\omega n} , 
  \end{equation}
which in view of~$\pi_k(\dn)=N_k(\tau_{2m})/n$ establishes~\eqref{l3a}.
\end{proof}
\begin{proof}[Proof of \refL{cl:NB}]
We start with the first identity in~\eqref{nb:pgf:pgf'},
  which is an immediate consequence of \eqref{eq:NB} and
$(1-px)^{-\alpha} = \sum_{r \in \NN} \binom{\alpha+r-1}{r} p^rx^r$. 
Differentiation then yields the second identity in~\eqref{nb:pgf:pgf'}, 
which for~$x=1$ yields~$\E Y = \alpha \cdot p/(1-p)$. 
Taking higher derivatives yields, for all integers~$k \ge 1$,
\begin{equation}\label{eq:factorial}
\E \prod_{0 \le j < k}(Y-j) = \biggpar{\frac{\dx^{k}}{\dx p^{k}} (1-p)^{-\alpha}} \cdot (1-p)^{\alpha}p^k = \alpha(\alpha+1)\cdots(\alpha+k-1) \cdot \Bigpar{\frac{p}{1-p}}^k . 
\end{equation}
Inserting~$\E Y = \alpha \cdot p/(1-p)$ 
into~\eqref{eq:factorial} then readily establishes~\eqref{nb:mom:mom1}, 
where the case~$k=0$ is trivial. 
Finally, \ref{cl:NB:mean}~follows immediately 
by substituting~$p=1-1/(1+x/\alpha)$ into~$\E Y = \alpha p/(1-p)$. 
\end{proof}
\begin{remark}\label{RNegBin}
We used above that the degrees can be obtained by stopping the
independent processes~$D_v(t)$ suitably. Another, related, construction
that also can be used to show asymptotic results is that, for any~$p\in(0,1)$,
if~$D_v\sim\NB(\ga,p)$ are independent, then
the degree sequence~$(d_v)_{v \in [n]}$ has the same distribution as~$(D_v)_{v \in [n]}$
conditioned on~$\sum_v D_v=2m$, see \eg{} \citet{Holst1979} or \citet[p.~624]{Pittel}.
That the degree distribution is asymptotically negative binomial can also
be shown by direct calculations \cite[Case~2b,~p.~153]{Polya1930}.
\end{remark}

\subsection{Proof of Theorem~\ref{thm:transfer}: relating~$G^{\alpha}_{n,m}$ with~$G^{\alpha,*}_{n,m}$}\label{sec:transfer}
For the transfer statements of Theorem~\ref{thm:transfer} and
Remark~\ref{rem:transfer}, 
the strategy is to compare the stepwise conditional probabilities of the added edges
in~$(G^{\alpha}_{n,i})_{i \ge 0}$ and its multigraph variant
$(G^{\alpha,*}_{n,i})_{i \ge 0}$;
these probabilities are identical up to the normalizing factors in the
denominator.
(This is the same strategy as in~\cite[Section~3]{Pittel}, although we do
the details differently and somewhat simpler.)
We begin by relating the two processes under an extra technical condition.
\begin{lemma}\label{lem:transfer:proc}
Define~$Z(G_i) := \sum_{v \in [n]} d_v(i)^3$, where~$d_v(i):=d_v(G_i)$ denotes the degree of vertex~$v$ in~$G_i$. 
Given $A,C,\alpha_0>0$,
the following holds 
whenever~$1 \le m \le C n$ and $\alpha \ge \alpha_0$. 
For any sequence~$(G_0, \ldots, G_m)$ of simple graphs 
with $V(G_i)=[n]$, $G_{i} \subset G_{i+1}$, $e(G_i)=i$, 
and~$Z(G_{m-1}) \le A n$, we have  
\begin{equation}\label{eq:transfer:proc:tech}
\Pr\bigpar{(G^{\alpha}_{n,i})_{0 \le i \le m} = (G_0, \ldots, G_m)} = \Theta(1) \cdot \Pr\bigpar{(G^{\alpha,*}_{n,i})_{0 \le i \le m} = (G_0, \ldots, G_m)} ,
\end{equation}
where the implicit constants in~\eqref{eq:transfer:proc:tech} may depend on~$C,\alpha_0,A$. 
\end{lemma}
\begin{proof}
Let~$\{v,w\} \in \binom{[n]}{2} \setminus E(G_i)$ denote the (unique) edge in which~$G_{i+1}$ and~$G_i$ differ. 
Note that, by construction, the edge~$\{v,w\}$ is added to~$G_i$ with probability proportional 
to~$2(d_v(i)+\alpha)(d_w(i)+\alpha)$ in both random graph processes, but each time the 
normalizing factor in the denominator differs slightly. 
Indeed, recalling~$\sum_{v \in [n]} d_v(i) =2i$, in the random 
multigraph process 
the normalizing factor equals
\begin{equation}\label{eq:norm:multi}
\begin{split}
& \sum_{\{x,y\} \in \binom{[n]}{2}} 2\bigpar{d_x(i)+\alpha} \bigpar{d_y(i)+\alpha} + \sum_{x \in [n]} \bigpar{d_x(i)+\alpha} \bigpar{d_y(i)+1 + \alpha} \\
& \qquad \qquad = \Bigpar{\sum_{x \in [n]} \bigpar{d_x(i)+\alpha}}^2 + \sum_{x \in [n]} \bigpar{d_x(i)+\alpha} = (2i+\alpha n)(2i + \alpha n + 1) ,
\end{split}
\end{equation}
whereas in the original (simple) random graph process the normalizing factor equals 
\begin{align}\label{eq:norm:simple}
&\sum_{\{x,y\} \in \binom{[n]}{2} \setminus E(G_i)} 2\bigpar{d_x(i)+\alpha}
\bigpar{d_y(i)+\alpha}
  \notag\\&\hskip4em
  = (2i + \alpha n)^2
- \biggsqpar{\underbrace{2 \sum_{\{x,y\} \in E(G_i)} \bigpar{d_x(i)+\alpha}
  \bigpar{d_y(i)+\alpha}
+  \sum_{x \in [n]} \bigpar{d_x(i)+\alpha}^2}_{=: Q(G_i)}} .
\end{align}
Putting things together, it follows that the one-step conditional probabilities 
are given by 
\begin{align}
\label{eq:Gi}
\Pr\Bigpar{ G^{\alpha}_{n,i+1}=G_{i+1}  \; \Big| \; (G^{\alpha}_{n,j})_{0 \le j \le i} = (G_0, \ldots, G_i)} & = \frac{2\bigpar{d_v(i)+\alpha}\bigpar{d_w(i)+\alpha}}{(2i + \alpha n)^2 - Q(G_i)} , \\
\label{eq:Hi}
\Pr\Bigpar{G^{\alpha,*}_{n,i+1}=G_{i+1} \; \Big| \; (G^{\alpha,*}_{n,j})_{0 \le j \le i} = (G_0, \ldots, G_i)} & = \frac{2\bigpar{d_v(i)+\alpha}\bigpar{d_w(i)+\alpha}}{(2i + \alpha n)(2i + \alpha n + 1)} .
\end{align}

Next we claim that~\eqref{eq:transfer:proc:tech} follows 
if~$\max_{0 \le i < m}Q(G_i) \le B \alpha^2 n$ for some constant~$B=B(C,\alpha_0,A)>0$. 
Indeed, the right-hand sides of~\eqref{eq:Gi}--\eqref{eq:Hi} 
are then equal up to a multiplicative factor of $1+O(n^{-1})$, 
where the implicit constants may depend on~$B$ and~$\alpha_0$.   
Noting that initially~$\Pr(G^{\alpha}_{n,0}=G_{0})=1=\Pr(G^{\alpha,*}_{n,0}=G_{0})$, 
then~\eqref{eq:transfer:proc:tech} follows readily by comparing the product 
of~$m \le Cn$ conditional probabilities.

It remains to prove~$\max_{0 \le i < m}Q(G_i) \le B \alpha^2 n$.
Using~$(x+\alpha)(y+\alpha) \le 4 \max\{x^2,y^2,\alpha^2\}$  
it follows that
\begin{equation}
\begin{split}
Q(G_i) 
& \le 8\sum_{\{x,y\} \in E(G_i)}\bigl(d_x(i)^2 + d_y(i)^2+\alpha^2\bigr) + 4\sum_{x \in [n]}\bigl(d_x(i)^2 +\alpha^2\bigr) \\
& \le 16 \sum_{v \in [n]} d_v(i)^3 + 8 \alpha^2 |E(G_i)|  + 4 \sum_{v \in [n]} d_v(i)^2 + 4 \alpha^2 n \le 20 \cdot \Bigsqpar{ Z(G_i) + \alpha^2 \max\{i,n\}} .
\end{split}
\end{equation}
Noting that~$Z(G_i) \le Z(G_{m-1})$ by monotonicity of the degrees, 
and using the assumptions $Z(G_{m-1})\le An$,
$i < m \le Cn$ and~$\alpha \ge \alpha_0$ 
we readily infer~$Q(G_i)\le 20 [A + \alpha^2 (C+1)] n \le B \alpha^2 n$ 
for suitable~$B=B(C,\alpha_0,A)>0$, completing the proof (as discussed). 
\end{proof}
\noindent 
To deduce the desired transfer statements, it intuitively remains to  
show that typically~$Z(G^{\alpha}_{n,m-1}) = O(n)$. 
Perhaps surprisingly, 
using Lemma~\ref{lem:transfer:proc} 
this can in fact be derived (or `bootstrapped')
from a corresponding bound for the more 
tractable process~$G^{\alpha,*}_{n,m}$ by a stopping time argument. 
\begin{proof}[Proof of Theorem~\ref{thm:transfer}]
Since~\eqref{eq:transfer:proc} implies~\eqref{eq:transfer}, 
it suffices to prove Remark~\ref{rem:transfer}. 
Recalling~\eqref{dxp:mux} and~\eqref{l3b} with~$\E Y_n^5=O(1)$, 
by Theorem~\ref{thm:degree}
(and monotonicity in $m$)
we may choose~$A = A(C,\alpha_0)> 0$
such that
\begin{align}\label{avesta}
\Pr(Z(\gganx{m-1})> A n) \le \Pr\bigpar{Z\bigpar{\gganx{\floor{Cn}}} > An } =o(1) . 
\end{align} 

Next, by summing~\eqref{eq:transfer:proc:tech} over all graph 
sequences from~$\cG_{n,m}$ with~$Z(G_{m-1}) \le A n$ 
(ignoring those sequences that are not realizable by the 
simple random graph process~$(G^{\alpha}_{n,i})_{0 \le i \le m}$) 
it follows that 
\begin{equation}\label{eq:transfer:seq}
\Pr( (G^{\alpha}_{n,i})_{0 \le i \le m} \in \cG_{n,m} \text{ and } Z(G^{\alpha}_{n,m-1}) \le A n) \le O(1) \cdot \Pr( (G^{\alpha,*}_{n,i})_{0 \le i \le m} \in \cG_{n,m}) , 
\end{equation}
where here and below we use the convention that 
all implicit constants may depend on~$C,\alpha_0,A$. 

Finally, we compare~$\Pr(Z(G^{\alpha}_{n,m-1})>A n)$ with~$\Pr(Z(G^{\alpha,*}_{n,m-1}) > A n)$. 
Since~$Z(G_i) < Z(G_{i+1})$ by monotonicity of the degrees, 
here the idea is to~focus on the first step
where~$Z(G_j) \le An$ is violated, and then only compare 
the probabilities in both processes up to (and including) that step 
via Lemma~\ref{lem:transfer:proc}. 
Turning to the details, define~$\cH_{n,j}$ as the set of all 
simple graph sequences~$(G_0,\ldots, G_j)$ 
with $V(G_i)=[n]$, $G_{i} \subset G_{i+1}$, 
$e(G_i)=i$, $Z(G_{j-1}) \le A n$ and $Z(G_{j}) > A n$. 
Noting that initially~$Z(G^{\alpha}_{n,0})=0=Z(G^{\alpha,*}_{n,0})$,
by summing~\eqref{eq:transfer:proc:tech} over all 
graph sequences from~$\cH_{n,j}$ it follows~that 
\begin{equation}\label{eq:Qsum}
\begin{split}
\Pr(Z(G^{\alpha}_{n,m-1}) > A n) 
& = \sum_{1 \le j \le m-1} \Pr( (G^{\alpha}_{n,i})_{0 \le i \le j} \in \cH_{n,j}) \\
& = \Theta(1) \cdot \sum_{1 \le j \le m-1} \Pr( (G^{\alpha,*}_{n,i})_{0 \le i \le j} \in \cH_{n,j}) 
\le O(1) \cdot \Pr(Z(G^{\alpha,*}_{n,m-1}) > A n) ,
\end{split}
\end{equation}
which together with~\eqref{avesta}--\eqref{eq:transfer:seq} completes 
the proof of inequality~\eqref{eq:transfer:proc}.
\end{proof}
\begin{remark}\label{rem:transfer:difference}
The transfer statement of~\refT{thm:transfer} fails when~$\alpha=\alpha(n) \to 0$ sufficiently fast:
e.g.,~for~$\alpha = o(n^{-2})$ and $m=O(n)$ it is easy to check 
that~$(G^{\alpha}_{n,i})_{0 \le i \le m}$ \whp{} sequentially builds up a~clique, 
whereas~$(G^{\alpha,*}_{n,i})_{0 \le i \le m}$ \whp{} only adds~loops. 
\end{remark}

\section{Size of the giant component: proof of the main theorems}\label{sec:L1} 
In this section we prove our main giant component results
Theorems~\ref{thm:main1}--\ref{thm:main} and \ref{Tmulti}, 
by closely following the heuristics from \refS{sec:heuristic}. 
(Recalling~$\alpha =\alpha(n)\to a \in (0,\infty]$ as~$n \to \infty$, 
we sometimes need to distinguish the cases~$a <\infty$ and~$a=\infty$,
since they have different asymptotic degree~distributions.)

\begin{proof}[Proof of \refT{Tmulti}]
  Recall that  \refT{thm:main1} is a special case of \refT{thm:main}.
Thus it suffices to prove that \refT{thm:main} holds with $L_j(m):=L_j(\gganm)$. 
%
By considering subsequences in the standard way, we may assume 
that~$\eps=\eps(n)\to\epsoo$, for some~$\epsoo\in[0,\infty)$, as~\ntoo.
Let $Y_n\sim\NB(\ga,\pn)$ be as in \refT{thm:degree},
and let $\gd_n:=n\qqw\go(n)$ with $\go(n):=\log n$, say 
(any sequence with~$n\qqw\ll\gd_n\ll n\qqqw$ would work).
  Then \eqref{l3a}--\eqref{l3b}
  show that the random vector
$\gd_n\qw\bigpar{(\pi_k(\dn)-\P(Y_n=k))_{k=0}^\infty,
  (\mux_k(\dn)-\E Y_n^k)_{k=0}^\infty}$
converges in probability to 0
in the product space~$\bbR^\infty\times\bbR^\infty$.
By the Skorohod coupling theorem~\cite[Theorem~4.30]{Kallenberg},
we may without loss of generality assume that the random multigraphs for
different $n$ are coupled such that
this holds \as, and thus
that, \as, for every for every integer~$k\ge0$ we have
  \begin{align}\label{lxa}
    \pi_k(\dn)& = \P(Y_n=k)+o\bigpar{\gd_n},\\
\label{lxb}
      \mux_k(\dn)& = \E Y_n^k +o\bigpar{\gd_n}.
    \end{align}

As a preparatory step, we now show that~$\dn=\dnx$ satisfies the
assumptions of \refT{thm:giant:conf} with
\begin{align}\label{winston}
  D\sim
  \begin{cases}
    \NB(a,\poo), & \text{if~$a<\infty$},
    \\
    \Po(1+\epsoo), & \text{if~$a=\infty$},
  \end{cases}
\end{align}
where~$\poo$ is defined as in~\eqref{magnus} below. 
Since~$m=\mc(1+\eps)$, recalling~\eqref{pn} and~\eqref{def:mc} we infer that, as~$n \to \infty$, 
\begin{align}\label{magnus}
    p_n
  =\frac{2\mc(1+\eps)}{n\ga+2\mc(1+\eps)}
  =\frac{(1+\eps)/(\ga+1)}{1+(1+\eps)/(\ga+1)}
      =\frac{1+\eps}{\ga+2+\eps}
\; \to \; \frac{1+\epsoo}{a+2+\epsoo} =: \poo.
\end{align}
If $a<\infty$, then it readily follows from~\eqref{eq:NB} that~$Y_n\dto \NB(a,\poo)$.
If instead $a=\infty$, then~\eqref{magnus} yields~$p_n\to0$ and~$\ga\pn\to1+\epsoo$ as~$n \to \infty$, so that~\eqref{eq:NB} implies
\begin{align}
  \P(Y_n=k)=\frac{\prod_{0\le j<k}(\ga p_n+jp_n)}{k!} (1-p_n)^\ga
  \; \to \;
	\frac{(1+\epsoo)^k}{k!}e^{-(1+\epsoo)},
  \qquad k\ge0,
\end{align}
and thus~$Y_n\dto \Po(1+\epsoo)$.
To sum up, in both cases we have~$Y_n\dto D$, so that~\eqref{lxa} implies, \as,
$\pi_k(\dn) = \P(Y_n=k)+o(1) \to\Pr(D=k)$ as~$n \to \infty$, for every~$k \ge 0$. 
Furthermore, $\E D\in(0,\infty)$ and $\P(D=1)>0$ are obvious 
(note that~$\poo>0$ when~$a<\infty$).
Moreover, \eqref{lxb} and \refT{thm:degree} yield, \as, $\mu_5(\dn)=\E Y_n^5+o(1) = O(1)$,
so the assumptions of \refT{thm:giant:conf} hold (as claimed).

Next, gearing up to apply \refT{thm:giant:conf} in the two separate cases~$\epsoo=0$ and~$\epsoo>0$,  
we now estimate~$\E D$ and~$\E D(D-2)$. 
In particular, \eqref{winston}, \eqref{nb:mom:mom1} and~\eqref{magnus} yield,
when~$a<\infty$,
\begin{align}\label{ed}
  \E D = a\frac{\poo}{1-\poo}=\frac{a}{a+1}(1+\epsoo)
  =\frac{1}{1+a\qw}(1+\epsoo),
\end{align}
which obviously holds for~$a=\infty$ too (with our convention~$\infty\qw=0$).
Similarly, \eqref{nb:mom:mom1} yields, when $a<\infty$,
\begin{align}\label{edd2}
  \E D(D-2) =\E D(D-1) - \E D = (\E D)^2 (1+a\qw) - \E D  
   =\frac{1}{1+a\qw}(1+\epsoo)\epsoo ,
\end{align}
which again holds for~$a=\infty$ by standard Poisson formulas for~$D \sim \Po(1+\epsoo)$.

\pfcase{$\eps(n)\to\epsoo>0$.}\label{case>}
In this case, \eqref{edd2} yields $\E D(D-2) >0$. Hence,
\refT{thm:giant:conf}\ref{thm:supercr} applies and shows the existence of 
a unique~$\xi=\xi(a,\epsoo)\in(0,1)$ such that
\begin{align}
  \label{padda}
  \E D\xi^{D-1} = \xi\E D.
\end{align}
Furthermore, by \refT{thm:equivalence} and the calculations above,
\as the degree sequence~$\dn$ is such that~\eqref{eq:supercr}
applies to~$\gganm$
conditioned on~$\dn$. Consequently,
\eqref{eq:supercr} holds for~$\gganm$ also unconditionally,
which is~\eqref{eq:thm:main:giant:L1}--\eqref{eq:thm:main:giant:L2}
with~$\rho_a(\eps)$ replaced by
\begin{align}\label{paddington}
  \rho_a(\epsoo):=1-\E \xi^D.
\end{align}
In the remainder we take~\eqref{paddington} together 
with~\eqref{magnus}, \eqref{winston} and~\eqref{padda} 
as the definition of the function~$\rho_a(\epsoo)$, 
for all~$a\in(0,\infty]$ and~$\epsoo\in(0,\infty)$.
We now study further properties of this function,
and for convenience temporarily write~$\eps$
instead of~$\epsoo$. (We no longer consider finite~$n$ here, so
there is no risk of~confusion.)

If~$a=\infty$, then by combining~$D\sim\Po(1+\eps)$ with standard 
Poisson formulas, we can rewrite~\eqref{padda} as
\begin{align}\label{ele}
e^{-(1+\eps)(1-\xi)}=\xi .
\end{align}
In view of~\eqref{paddington} this also yields the identity 
\begin{align}\label{erika}
\rho_\infty(\eps)=1-e^{-(1+\eps)(1-\xi)}=1-\xi,
\end{align}
establishing that~$\rho_\infty(\xi)$ is the positive solution of \eqref{rhooo}, as asserted.

If~$a<\infty$, 
then by combining~$D \sim \NB(a,\poo)$ with~\eqref{nb:pgf:pgf'} and \eqref{magnus} 
we can rewrite~\eqref{padda} as
\begin{align}\label{sofie}
   \lrpar{\frac{a+1}{a+2+\eps-(1+\eps)\xi}}^{a+1}
  =\xi.
\end{align}
Then \eqref{paddington} and \eqref{nb:pgf:pgf'} yield
\begin{align}\label{emma}
\rho_a(\eps)=1-   \lrpar{\frac{a+1}{a+2+\eps-(1+\eps)\xi}}^{a}
  =1-\xi^{a/(a+1)}.
\end{align}
By \eqref{emma}, \eqref{sofie} and elementary algebra, 
it also follows that 
\begin{equation}\label{eq:NB:rho+}
  \rho_a(\eps)
  = 1-\left(1+ \frac{(1+\eps)(1-\xi)}{a+1}\right) \cdot \xi
  = (1-\xi) \cdot \left(1-\frac{(1+\eps) \xi}{a+1}\right) .
\end{equation}

We remark that, in view of~\eqref{erika}, 
equations~\eqref{emma}--\eqref{eq:NB:rho+} both also hold when~$a=\infty$ 
(by interpreting the fractions in the natural way, \ie, as limits).
As a further technical interlude (that can safely be skipped on a first reading), note that 
the difference between the two sides in \eqref{sofie} is a convex function
of $\xi\in[0,1]$, which vanishes at $\xi=\xi(a,\eps)$ and $\xi=1$ but not
identically. Hence its derivative at $\xi(a,\eps)$ must be negative, and
thus the implicit function theorem applies and shows that~$\xi(a,\eps)$ 
is an analytic function of $(a,\eps)\in(0,\infty)^2$.
Hence, using~\eqref{emma} or~\eqref{eq:NB:rho+}, 
it follows that~$\rho_a(\eps)$ is also an analytic function of~$(a,\eps)$.


\pfcase{$\eps(n)\to\epsoo=0$.} \label{case0}
In this case,~\eqref{edd2} yields~$\E D(D-2)=0$. 
Using~\eqref{ed} and~\eqref{nb:mom:mom1} when~$a< \infty$ and standard Poisson formulas for~$D \sim \Po(1)$ when~$a=\infty$, it follows that~$\E D = 1/(1+a\qw)$ and 
\begin{align}\label{eddd}
\frac{\E D}{\E D(D-1)(D-2)} = \frac{1}{(\E D)^2 (1+1/a)(1+2/a)}
= \frac{1+a\qw}{1+2/a} .
\end{align}
Moreover, using~\eqref{lxb} and~$\gd_n \ll n\qqqw\ll \eps$, 
 in analogy with~\eqref{ed}--\eqref{edd2} it follows that, \as, 
\begin{align}\label{jw}
  \muu_n/n= \mux_2(\dn)-2\mux_1(\dn)
	= \E Y_n(Y_n-2) + o(\gd_n)
  =   \frac{1}{1+\ga\qw}(1+\eps)\eps + o(\eps)
    \sim   \frac{1}{1+a\qw}\eps.
\end{align}
Consequently, $n^{-2/3}\muu_n = \Theta(\eps n\qqq) \to\infty$,
and \refT{thm:giant:conf}\ref{thm:critical} applies \as{}
to~$\gganm$ conditioned on $\dn$,
again using \refT{thm:equivalence}.
Hence, \eqref{eq:critical} yields,
conditioned on $\dn$ and therefore also unconditioned, 
\begin{align}\label{tho}
  L_1(\gganm)=\lrpar{\frac{2(1+a\qw)}{1+2/a}+\op(1)}\muu_n
    =\frac{2\eps n}{1+2/a} \cdot \bigpar{1+\op(1)}.
\end{align}
We similarly also obtain~$L_2(\gganm)=\op(\muu_n)=\op(\eps n)=\op(L_1(\gganm))$. 
This verifies \eqref{eq:thm:main:giant:L2}, in this case too.
However, to derive \eqref{eq:thm:main:giant:L1} from \eqref{tho}, it remains 
to stitch the two formulas together by showing that, as~$\eps\to0$ with~$a\le\infty$ fixed,
we have~$\rho_a(\eps)\sim 2\eps/(1+2/a)$, as asserted (more precisely) in
\eqref{eq:rho:linear}. 
We shall do this by an analytic argument. (An alternative, more conceptual,
proof is sketched in \refR{Rfit}.)
Since the claimed asymptotics is well-known in the Poisson case~$a=\infty$, 
for simplicity we henceforth assume~$a<\infty$ 
(the proof in the simpler case~$a=\infty$ proceeds in the same way).  
With foresight, we define 
\begin{align}
  \label{psi}
\psi:=(1+\eps)(1-\xi)>0  ,
\end{align}
and then (similarly to the calculations leading to~\eqref{eq:NB:rho+} above) rewrite~\eqref{sofie} as
\begin{equation}\label{eq:NB:xi}
\xi
= \left(1+ \frac{(1+\eps)(1-\xi)}{a+1}\right)^{-(a+1)}
= \left(1+ \frac{\psi}{a+1}\right)^{-(a+1)}.
 \end{equation}
 Using~\eqref{sofie} and~$\log(1+x)\le x$,
 it follows that 
\begin{align}\label{kk}
  (1-\xi) + (1-\xi)^2/2 < -\log\bigpar{1-(1-\xi)}
  = \log \frac{1}{\xi}
=(a+1)\log\left(1+ \frac{\psi}{a+1}\right)
  \le \psi=(1+\eps)(1-\xi), 
\end{align}
which after dividing by~$1-\xi>0$ easily gives $1-\xi < 2\eps$, and
thus~$\rho\le 1-\xi < 2\eps$ by~\eqref{emma} or~\eqref{eq:NB:rho+}.
Furthermore, equations~\eqref{psi}--\eqref{eq:NB:xi} and a Taylor expansion (or the general binomial series) yield
\begin{align}
  \frac{\psi}{1+\eps}
  =1-\xi
  = 1-\left(1+ \frac{\psi}{a+1}\right)^{-(a+1)}
  = \psi-\frac{a+2}{2(a+1)} \psi^2 + O(\psi^3).  
\end{align}
Since~$0 < \psi \le (1+\eps)2\eps = O(\eps)$,
after dividing by~$\psi>0$ 
it then routinely follows that  
\begin{align}\label{klm}
  \psi= \frac{2(a+1)\eps}{(a+2)(1+\eps)} + O(\eps^2) 
  = \frac{2(a+1)}{a+2}\eps + O(\eps^2) .
  \end{align}
Recalling \eqref{psi} and that~$1-\xi= O(\eps)$, 
now~\eqref{eq:NB:rho+} and \eqref{klm}
imply for~$\eps \searrow 0$ the asymptotics 
\begin{align}
  \rho_a(\eps) = \frac{\psi}{1+\eps} \cdot \left(\frac{a}{a+1}+O(\eps)\right) 
= \frac{2a\eps}{a+2} + O(\eps^2) ,
\end{align}
as stated in \eqref{eq:rho:linear}. 
This completes the proof of \refT{Tmulti}, as discussed below~\eqref{tho}.
\end{proof}

\begin{proof}[Proof of Theorems \ref{thm:main1} and \ref{thm:main}]
These results follow from the multigraph version \refT{Tmulti} and 
the transfer result \refT{thm:transfer} by routine arguments 
(the key point is that any event which fails with probability at most~$\pi=o(1)$ in~$\gganm$ 
fails with probability at most $B \cdot \pi+o(1)=o(1)$ in~$\ganm$).
\end{proof}

\begin{remark}\label{rem:NB:rho}
When~$\eps \to \infty$ (with $\ga$ fixed),
it follows easily from~\eqref{eq:NB:xi} that~$\xi \to 0$,
and thus \eqref{emma} implies~$\rho_a(\eps) = 1- \xi^{a/(a+1)}\to 1$. 
When~$a \to \infty$ with $\eps \in (0,\infty)$ fixed,
it follows easily from \eqref{eq:NB:xi} that $\xi(a,\eps)\to\xi(\infty,\eps)$
satisfying \eqref{ele}, and thus \eqref{eq:NB:rho+} and \eqref{erika} imply
$\rho_{a}(\eps) \to 1-\xi(\infty,\eps)=\rho_{\infty}(\eps)$.
Furthermore, since $\rho_\infty(\eps)$ is continuous with bounded range,
and each $\rho_a(\eps)$ is increasing, it follows 
that~$\rho_a(\eps)\to\rho_\infty(\eps)$ uniformly in~$\eps\in(0,\infty)$.
We omit the details.
\end{remark}

\begin{remark}\label{Rfit}
  An alternative proof of \eqref{eq:rho:linear} without calculations
  is based on the somewhat vague but very general idea
  that  when we have a limit theorem with different cases as above, then
  the cases have to fit together smoothly.
  To see this in the present case,
  fix $\ga$ and
write for simplicity $\G(n,\eps):=\gganx{\floor{\mc(1+\eps)}}$.
Consider a sequence $\eps_i\to0$.
First, fix $i$ and take $\eps=\eps_i$ constant.
Then, by Case \ref{case>} ($\epsoo>0$) 
in the proof above, $L_1(\G(n,\eps_i))=\rho_\ga(\eps_i) n
\bigpar{1+\op(1)}$. Hence we can choose $n_i$ so large that with probability
$>1-2^{-i}$, 
\begin{align}\label{mas}
  |L_1(\G(n_i,\eps_i))/n_i-\rho_\ga(\eps_i)|<2^{-i}\eps_i.
\end{align}
We may further assume $n_{i+1}>n_i$ and $n_i> i \eps_i^{-3}$.
Now consider the sequence of multigraphs $\G(n_i,\eps_i)$, $i\ge1$.
Case \ref{case0} ($\epsoo=0$) in the proof
above applies, so \eqref{tho} holds  for the multigraphs $\G(n_i,\eps_i)$,
which together with \eqref{mas} implies
\eqref{eq:rho:linear}, up to a weaker error term~$o(\eps)$.

It is also possible to recover the error term~$O(\eps^2)$ without explicit
calculations. Suppose for definiteness that~$a<\infty$. We may then solve
\eqref{sofie} for~$\eps$, and obtain~$\eps=\eps(\xi)$ as an analytic function
of~$\xi$ in some complex neighbourhood of~$\xi=1$. The derivative at~$1$ is
non-zero, \eg{}~as a consequence of~\eqref{eq:rho:linear} and~\eqref{emma},
and thus the implicit function theorem shows that $\xi(a,\eps)$ may be
extended to an analytic function of $\eps$ in some neighbourhood of~$0$.
By~\eqref{emma}, the same holds for~$\rho_\ga(\eps)$. Thus~$\rho_\ga(\eps)$
is analytic on the closed half-line~$[0,\infty)$, and all derivatives stay
bounded as~$\eps\to0$; so the error term~$O(\eps^2)$ 
in~\eqref{eq:rho:linear} follows by Taylor's~theorem. 
\end{remark}

\section{Final remarks}\label{sec:final}

\subsection{Preferential attachment with negative~$\ga$}\label{Snegative}
As in previous work, 
we have throughout assumed~$\ga>0$. 
It is also possible to consider~$\ga<0$, provided~$\ga=-r$ for some integer~$r=|\ga| \in \NN$.
In this case, the process adds edges by choosing vertices
as in \refS{sec:equivalence}
with probability proportional to~$r-d_v$, where~$d_v$ is the current degree.
This is equivalent to starting with~$r$~half-edges for each vertex, and then
choosing pairs of half-edges uniformly at random, in the simple graph
version conditioned on keeping the graph simple.
The multigraph version can be seen (at least when $rn$ is even)
as the stepwise construction of an $r$-regular multigraph by the configuration
model.
Equivalently, it is the process obtained if we construct an~$r$-regular
multigraph by the configuration model and then take its edges in uniformly
random order; hence it can be seen as edge percolation on this random multigraph. 

This process was introduced by \citet{SW1999} as a method to generate almost
uniformly distributed regular graphs,
see also \cite{KV2004,KV2006}. 
Note that no vertex will ever get degree more than $r$, so the process stops
when we cannot add any edge without increasing the maximum degree above
$r$. Then there are $n-O(1)$ vertices of degree $r$, in both the simple
graph and multigraph version, 
so the final (multi)graph is almost regular.
Thus, the process stops after~$m=\mstop-O(1)$ steps, where
\begin{align}
  \label{def:mstop}
  \mstop:=\frac{rn}{2} = \frac{|\ga|}2n.
\end{align}

The papers~\cite{SW1999,KV2004,KV2006} just mentioned are mainly interested
in the resulting final graph, but we may also consider the entire process, and, in
particular, the emergence of the giant component in this process. 
In the following discussion we shall always assume that~$r:=|\ga|\ge 3$, 
so that~$n/2 < \mc < n < \mstop$ by~\eqref{def:mc} and~\eqref{def:mstop}.  
Fix a small~$\eta>0$ with~$(1-\eta)\mstop>\mc$.
Assuming~$m\le(1-\eta)\mstop$, it then is fairly easy to verify that the proofs 
from Sections~\ref{sec:approx:proofs}--\ref{sec:L1} carry over to that case 
(with minor routine changes).
The only noteworthy difference is that in \refS{sec:degree} the processes~$D_v(t)$
are now  death processes, starting with $r$~particles that
each dies at rate~$1$, and it turns out that~\eqref{eq:Dvt} is then replaced by
\begin{align}
  D_v(t)\sim\Bin\bigpar{|\ga|,1-e^{-t}}.
\end{align}
As a consequence, the asymptotic degree distribution is now binomial, and it follows that
\refT{thm:degree} holds with $Y_n\sim\Bin(|\ga|,p_n)$ and~$p_n:=2m/(n|\ga|)$.
In spite of these differences, it turns out that the probability generating function of~$Y_n$ 
can in the case~$\alpha<0$
be written in the same form
\begin{align}
  \E x^{Y_n} = \Bigpar{1+\frac{2m}{n\ga}(1-x)}^{-\ga} .
\end{align}
as in the case~$\ga>0$ (see equations~\eqref{nb:pgf:pgf'} and~\eqref{pn},
where~$(1-p_nx)/(1-p_n)=1+p_n(1-x)/(1-p_n)$ and $p_n/(1-p_n)=2m/(n\ga)$ hold).  
As a consequence, the equations~\eqref{sofie}--\eqref{eq:NB:rho+} that 
define~$\rho_\ga(\eps)$ are still valid.
Hence, assuming~$r\ge3$, Theorems~\ref{thm:main1}, \ref{thm:main} and \ref{Tmulti} hold for~$\ga:=-r$ too
(the assumption~$m\le(1-\eta)\mstop$ may be eliminated by
the same argument as in \refR{Rnobound}).
We may also let~$\ga\to-\infty$, with~$\rho_{-\infty}=\rho_\infty$; 
we leave the routine details to the reader.  
(In the multigraph case,
these results 
are special cases of earlier results on edge
percolation in random multigraphs with given vertex degrees,
see~\cite{Fountoulakis,SJ215}.)

\subsection{Extensions and variants}
\subsubsection{Hypergraphs and different vertex weights} 
It seems possible to adapt the methods of this paper to study 
hypergraph variants of the preferential attachment random graph process.
Similarly, it also seems possible to analyze the natural variant where a 
new edge~$\{v,w\}$ is added with probability proportional to~$(d_v+\alpha_v)(d_w+\alpha_w)$, 
i.e., where each vertex has its own `preference weight' 
(though here some assumptions on the~$(\alpha_v)_{v \in [n]}$ seem necessary  
in order to prove a multigraph transfer statement akin to \refT{thm:transfer}). 
We leave these extensions to the reader.

\subsubsection{More general preferential attachment functions}\label{sec:final:f} 
As suggested by Lov\'asz, see~\cite{Pittel}, and partly already by Erd{\H{o}}s~and~R{\'e}nyi~\cite{ER1961}, 
it is natural to study processes with more general preferential 
attachment functions~$f=f_n: \NN \to [0,\infty)$. 
We write~$(G^{f}_{n,m})_{m \ge 0}$ for the variant of~$(G^{\alpha}_{n,m})_{m \ge 0}$ where
the next edge connects two currently non-adjacent~$v$ and~$w$ with probability proportional to~$f(d_v) f(d_w)$. 
Similarly, we write~$(G^{f,*}_{n,m})_{m \ge 0}$ for the multigraph variant 
where the next edge connects distinct~$v$ and~$w$ with probability proportional to~$2f(d_v) f(d_w)$, 
and forms a loop at~$v$ with probability proportional to~$f(d_v) f(d_v+1)$. 
This general class of dynamic network models not only contains the
preferential attachment process (via~$f(k):=k+\alpha$)
but also the classical Erd{\H{o}}s--R{\'e}nyi process (via~$f(k):=1$), 
the configuration model for $d$-regular graphs
and the Steger--Wormald process \cite{SW1999} discussed in \refS{Snegative}
(via~$f(k):=\max\{d-k,0\}$),
and the random $d$-process~\cite{RWd,WWd} (via~$f(k) := \indic{k < d}$). 

Here the \emph{problem of determining the asymptotic degree distribution} after~$m$ steps is feasible
for many functions~$f$ via the pure birth-process based construction of~$(G^{f,*}_{n,m})_{m \ge 0}$ from \refS{sec:degree}, 
with birth-rates~$\lambda_k:=f(k)$. 
For example, using this construction they are easily seen to be
asymptotically Poisson in the Erd{\H{o}}s--R{\'e}nyi process, 
and truncated Poisson in the random~$d$-process.  
By contrast, the \emph{giant component problem} for general functions~$f$ appears to be more challenging.  
The crux is that the conditional equivalence argument from \refS{sec:equivalence} 
(which crucially allowed us to work with the configuration model) 
seemingly only carries over to linear functions,
which motivates the following conceptually interesting~problem. 
\begin{problem}%
Study the giant component problem for~$G^{f}_{n,m}$ or~$G^{f,*}_{n,m}$ when~$f$ is non-linear. 
\end{problem}

\subsubsection{Edge-rewiring variant without edge-growth}\label{sec:final:rewiring}
There is a natural variant of the preferential attachment random multigraph process 
where the number of edges~$m=\Theta(n)$ is also fixed; 
this was proposed in the complex networks literature by Dorogovtsev, Mendes and Samukhin around~2002 
(see~\cite[Section~3]{DMS2003} and~\cite[Chapter~4.3]{DM2003}). 
To be more precise, this preferential attachment edge-rewiring process 
starts with a given (arbitrary) initial multigraph~$G_{0}$ with vertex set~$[n]$ and~$m$~multiedges, 
and then proceeds stepwise as follows: 
a uniform endvertex~$v$ of a uniformly chosen edge~$e$ is selected, and then~$e$ is replaced with the edge~$\{v,w\}$, 
where~$w$ is chosen with probability proportional to~$d_w+\alpha$. 
This rewiring process (also called simple edge-selection process~\cite{HP2010} or edge reconnecting model~\cite{RS2012}) 
converges rapidly to a unique stationary distribution~$G^{\alpha,\infty}_{n,m}$ 
(see~\cite[Sections~1.1--1.2]{HP2010}, \cite[Section~2.2]{RS2012}, and the `equilibrium' discussion in~\cite{DMS2003,DM2003}), 
which in fact 
has the same distribution as~$G^{\alpha,*}_{n,m}$ (see~\cite[Lemma~2.1]{RS2012}).
It follows that we can use the random multigraph~$G^{\alpha,*}_{n,m}$ to derive the long-run asymptotic properties 
of the edge-rewiring process with~$n$ vertices and~$m$ edges. 
In particular,
the result by \citet{Pittel} (or our \refT{thm:main1} and~\ref{Tmulti}) confirms
Conjecture~4.3 of Hruz and Peter~\cite{HP2010} regarding the existence of a
giant~component.

\subsection{Other properties}\label{sec:other}
The method used in this paper, where the auxiliary results from
\refS{sec:overview} are combined with results for~$\Gxdn$,  
can also be used to study other properties of of~$\ganm$ and~$\gganm$.
We briefly mention a few selected examples of interest here (leaving details to the reader).

First, we consider the \emph{number of vertices in `small' components of size~$k=O(1)$}. 
Under the assumptions of \refT{thm:giant:conf}, in~$\Gxdn$ their typical number can easily be related via~\cite[Lemma~4.1]{SJ241} to a branching process~$\bp$
that depends on (a size-biased version of) the `idealized degree distribution'~$D$ from \refT{thm:degree}. 
A routine analysis of~$\Pr(|\bp|=k)$ then leads to a conceptually simple proof of the counting formula~\cite[Lemma~4(b)]{Pittel}.\footnote{Pittel's result~\cite[Lemma~4(b)]{Pittel} is for the expected number of tree-components, and allows for growing~$k$.} 

Secondly, the so-called \emph{susceptibility}, which is 
the expected component-size of a randomly chosen vertex,
can be similarly estimated via~$\E |\bp|$ in the `subcritical' case, see~\cite[Corollary~3.2]{SJ241} and \refApp{sec:mc:heuristic}.

Thirdly, the widely studied \emph{$k$-core} is the largest induced subgraph
with minimum vertex degree at least~$k$ (which can be empty).  
By mimicking the first part of \refT{Tmulti}, 
one can use~\cite[Theorem~2.3]{JL2007} with~$D$ from \refT{thm:degree}
to show that, for each~$k \ge 2$, a linear-sized $k$-core whp appears in~$\ganm$ and~$\gganm$ around~$c_k n$ steps, 
where~$c_k := \tfrac{1}{2}\inf_{\mu>0} \mu/\Pr(Z_{\alpha}(\mu) \ge k-1)$ with negative binomial random variable~$Z_{\alpha}(\mu) \sim \NB\bigpar{\alpha+1,\mu/(\alpha+\mu)}$. 

Finally, we emphasize that our methods are geared towards the spare case~$m=O(n)$. 
Indeed, for denser graphs the reduction of the simple graph process to the multigraph variant
breaks down, at least in the present form (see \refT{thm:transfer}). 
This limitation is not just a mere proof artifact: in the multigraph process the threshold for connectivity is located around~$n^{1+\alpha^{-1}}$ edges~\cite[Theorem~2]{Pittel}, 
so for~$\alpha<1$ it must differ from the connectivity threshold of the simple graph process (whose location remains an open problem).

\subsection{Further open problems}
We close with some open problems for 
the preferential attachment process~$(G^{\alpha}_{n,m})_{m \ge 0}$ 
phase transition, which are all inspired by the corresponding 
behaviour of the Erd{\H{o}}s--R{\'e}nyi reference model: 
\begin{problem}
In the \emph{subcritical phase}~$m=\mc(1-\eps)$ with~$\eps=o(1)$ and $\eps^3n \to \infty$, 
	show that whp~$L_1(\mc(1-\eps)) \sim C_{\alpha}\eps^{-2}\log(\eps^3n)$ 
	for some constant~$C_{\alpha}>0$ (sharpening~\cite{Pittel}), 
	and prove a variant of \refCr{cor:main}. 
\end{problem}
\begin{problem}
In the \emph{critical window}~$m=\mc(1+\eps)$ with~$|\eps|=O(n^{-1/3})$, 
	writing $L_j=L_j(G^\alpha_{n,m})$ for the size of the $j$th largest component of~$G^\alpha_{n,m}$, 
	show that the sequence~$(L_j/n^{2/3})_{j\ge1}$ converges to a limiting distribution.
(\cite{Pittel} establishes that~$L_j$ is of order~$n^{2/3}$, for any fixed~$j \ge 1$.) 
	For the random multigraph~$\gganm$ this can be shown by combining our conditioning approach with the configuration model results from~\cite{DharaEtAl}, 
	but it remains open for~$\ganm$ (as \refT{thm:transfer} does not permit transfer of convergence in~distribution).
	Also, does the resulting limiting distribution equal the corresponding \ER{} one when~$\alpha = \omega(n^{2/3})$? 
\end{problem}
\begin{problem}
In the \emph{supercritical phase}~$m=\mc(1+\eps)$ with~$\eps=o(1)$ and $\eps^3n \to \infty$, show that~\eqref{eq:thm:main:giant:L2} can be improved to whp~$L_2(\mc(1+\eps)) \le D_{\alpha} \eps^{-2}\log(\eps^3n)$ for some constant~$D_{\alpha}>0$. 
\end{problem}

\bigskip{\noindent\bf Acknowledgements.}  
Part of this work was carried out during the authors' visit to the 
Isaac Newton Institute for Mathematical Sciences during the programme 
Theoretical Foundations for Statistical Network Analysis 
(EPSCR Grant Number EP/K032208/1).

\small
\bibliographystyle{plain}

\begin{thebibliography}{10}


\bibitem{AN1972}
K.B.~Athreya and P.E.~Ney.
\newblock \emph{Branching Processes}. 
Springer-Verlag, Berlin (1972) 
  
\bibitem{Barabasi2016}
A.-L.~Barab{\'a}si.
\newblock {\em Network Science}.
\newblock Cambridge University Press (2016).

\bibitem{BA1999}
A.-L.~Barab{\'a}si and R.~Albert.
\newblock Emergence of scaling in random networks.
\newblock {\em Science} {\bf 286} (1999), 509--512.

\bibitem[Ben-Naim and Krapivsky(2012)]{BNK}
E.~Ben-Naim and P.L.~Krapivsky.
\newblock Popularity-driven networking.
\newblock {\em EPL (Europhysics Letters)} {\bf 97} (2012), 48003.

\bibitem{BBRG}
B.~Bollob{\'a}s.
\newblock {\em Random Graphs}.
\newblock 2nd~ed., Cambridge University Press (2001). 

\bibitem{BR2009}
B.~Bollob{\'a}s and O.~Riordan.
\newblock Random graphs and branching processes.
\newblock In {\em Handbook of Large-Scale Random Networks}, 
\newblock Bolyai Soc.\ Math.\ Stud {\bf 18} (2009), pp.~15--115.

\bibitem{BR2013}
B.~Bollob{\'a}s and O.~Riordan.
\newblock The phase transition in the {E}rd\H{o}s-{R}\'{e}nyi random graph process.
\newblock In {\em Erd\"{o}s Centennial}, 
\newblock Bolyai Soc.\ Math.\ Stud {\bf 25} (2013), pp.~59--110.

\bibitem{BRST2001}
B.~Bollob\'{a}s, O.~Riordan, J.~Spencer, and G.~Tusn\'{a}dy.
\newblock The degree sequence of a scale-free random graph process.
\newblock {\em Rand.\ Struct.\ \& Algor.} {\bf 18} (2001), 279--290.

\bibitem[Borgs, Chayes, Lov\'{a}sz, S\'{o}s, and Vesztergombi(2011)]{BCLSV2011}
C.~Borgs, J.~Chayes, L.~Lov\'{a}sz, V.~S\'{o}s, and K.~Vesztergombi.
\newblock Limits of randomly grown graph sequences.
\newblock {\em European J.\ Combin.} {\bf 32} (2011), 985--999.

\bibitem[Dhara et al(2017)]{DharaEtAl}
S.~Dhara, R.~van der Hofstad, J.S.H.~van Leeuwaarden and S.~Sen.
\newblock Critical window for the configuration model: finite third moment degrees.
\newblock \emph{Electron.\ J.\ Probab.} \textbf{22} (2017), Paper~16, 33~pp. 

\bibitem{DM2003}
S.N.~Dorogovtsev and J.F.F.~Mendes.
\newblock {\em Evolution of Networks}.
\newblock Oxford University Press (2003).

\bibitem{DMS2003}
S.N.~Dorogovtsev, J.F.F.~Mendes, and A.N.~Samukhin.
\newblock Principles of statistical mechanics of uncorrelated random networks.
\newblock {\em Nuclear Phys.~B} {\bf 666} (2003), 396--416.

\bibitem{Durrett2010}
R.~Durrett.
\newblock {\em Random Graph Dynamics}.
\newblock Cambridge University Press (2010).

\bibitem{ER1960}
P.~Erd\H{o}s and A.~R\'enyi. 
\newblock On the evolution of random graphs. 
\newblock {\em Magyar Tud.\ Akad.\ Mat.\ Kutat\'o Int.\ K\"ozl} {\bf 5} (1960), 17--61.

\bibitem{ER1961}
P.~Erd\H{o}s and A.~R\'enyi. 
\newblock On the evolution of random graphs.
\newblock {\em Bull.\ Inst.\ Internat.\ Statist.}, {\bf 38} (1961), 343--347.

\bibitem[Fountoulakis(2007)]{Fountoulakis}
N.~Fountoulakis.
\newblock Percolation on sparse random graphs with given degree sequence.
\newblock \emph{Internet Math.} \textbf4 (2007), 329--356. 

\bibitem{FKRG}
A.~Frieze and M.~Karo{\'n}ski.
\newblock {\em Introduction to Random Graphs}.
\newblock Cambridge University Press (2016).

\bibitem{GS2001}
G.~Grimmett and D.~Stirzaker.
\newblock {\em Probability and Random Processes}.
\newblock 3rd~ed., Oxford University Press (2001).

\bibitem{Hofstad2017}
R.~van~der~Hofstad.
\newblock {\em Random Graphs and Complex Networks:~Vol.~1}.
\newblock Cambridge University Press (2017).

\bibitem{SJ313}
R.~van~der~Hofstad, S.~Janson and M.~Luczak.
\newblock Component structure of the configuration model:  barely supercritical case. 
\newblock {\em Rand.\ Struct.\ \& Algor.}, Early View (2019).


\bibitem{vdHN2012}
R.~van~der~Hofstad and A.~Nachmias.
\newblock {Hypercube percolation}.
\newblock  {\em J.~Eur.~Math.~Soc.} {\bf 19} (2017), 725--814.

\bibitem[Holst(1979)]{Holst1979}
L.~Holst.
\newblock A unified approach to limit theorems for urn models.
\newblock \emph{J.\ Appl.\ Probab.} {\bf 16} (1979), 154--162. 

\bibitem{HP2010}
T.~Hruz and U.~Peter.
\newblock Nongrowing preferential attachment random graphs.
\newblock {\em Internet Math.} {\bf 6} (2010), 461--487.

\bibitem{HN2016}
T.~Hulshof and A.~Nachmias.
\newblock Slightly subcritical hypercube percolation.
\newblock Preprint~(2016). \texttt{arXiv:1612.01772}.

\bibitem[Janson(2004)]{SJ154}
S.~Janson.
\newblock Functional limit theorems for multitype branching processes and generalized P\'olya urns. 
\newblock \emph{Stochastic Process.\ Appl.} \textbf{110} (2004),  177--245. 

\bibitem[Janson(2006)]{SJ169}
S.~Janson.
\newblock Limit theorems for triangular urn schemes.
\newblock \emph{Probab.\ Theory Related Fields} \textbf{134} (2006),  417--452. 

\bibitem[Janson(2009)]{SJ215}
S.~Janson.
\newblock On percolation in random graphs with given vertex degrees.
\newblock \emph{Electron.\ J.\ Probab.} \textbf{14} (2009), 87--118. 

\bibitem[Janson(2010)]{SJ241}
S.~Janson.
\newblock Susceptibility of random graphs with given vertex degrees.
\newblock \emph{J.\ Comb.} \textbf1 (2010),  357--387. 

\bibitem[Janson(2009+)]{SJN6}
S.~Janson.
\newblock Probability asymptotics: notes on notation.
\newblock Preprint (2011).
\texttt{arXiv:1108.3924}    

\bibitem[Janson(2018)]{Jan2018}
S.~Janson.
\newblock On edge exchangeable random graphs.
\newblock {\em J.\ Stat.\ Phys.} {\bf 173} (2018), 448--484.

\bibitem{JL2007}
S.~Janson and M.~Luczak.
\newblock A simple solution to the {$k$}-core problem.
\newblock {\em Rand.\ Struct.\ \& Algor.}~{\bf 30} (2007), 50--62.

\bibitem{JL2009}
S.~Janson and M.~Luczak.
\newblock A new approach to the giant component problem.
\newblock {\em Rand.\ Struct.\ \& Algor.}~{\bf 34} (2009), 197--216.

\bibitem{JLR}
S.~Janson, T.~{\L}uczak and A.~Ruci{\'n}ski.
\newblock {\em Random Graphs}.
\newblock Wiley-Interscience (2000).

\bibitem{JW2016}
S.~Janson and L.~Warnke.
\newblock On the critical probability in percolation.
\newblock {\em Electron.\ J.\ Probab.} {\bf 23} (2018), Paper~1.

\bibitem{JohnsonKotz}
N.L.~Johnson and S.~Kotz.
\newblock \emph{Urn Models and Their Applications}.
\newblock Wiley, New York (1977).
  
\bibitem{Kallenberg}
O.~Kallenberg.
\newblock \emph{Foundations of Modern Probability.}
2nd ed., Springer, New York (2002).

\bibitem{KV2004}
J.H.~Kim and V.H.~Vu.
\newblock Sandwiching random graphs: universality between random graph models.
\newblock {\em Adv.~Math.}~{\bf 188} (2004), 444--469.

\bibitem{KV2006}
J.H.~Kim and V.H.~Vu.
\newblock Generating random regular graphs.
\newblock \emph{Combinatorica}, \textbf{26} (2006), 683--708.

\bibitem{KRBN2010}
P.L.~Krapivsky, S.~Redner, and E.~Ben-Naim.
\newblock {\em A Kinetic View of Statistical Physics}.
\newblock Cambridge University Press (2010).

\bibitem[Markov(1917)]{Markov1917} 
A.A.~Markov.
Sur quelques formules limites du calcul des probabilit\'es. (Russian.)
\emph{Bulletin de l'Acad\'emie Imp\'eriale des Sciences}
\textbf{11} (1917), 177--186.

\bibitem[Molloy and Reed(1998)]{MR1998}
M.~Molloy and B.~Reed.
\newblock The size of the giant component of a random graph with a given degree sequence.
\newblock {\em Combin.\ Probab.\ Comput.} {\bf 7} (1998), 295--305.

\bibitem{NP2010}
A.~Nachmias and Y.~Peres.
\newblock Critical percolation on random regular graphs.
\newblock {\em Rand.\ Struct.\ \& Algor.} {\bf 36} (2010), 111--148.

\bibitem[Pittel(2010)]{Pittel}
B.~Pittel.
\newblock On a random graph evolving by degrees.
\newblock {\em Adv.\ Math.}~{\bf 223} (2010), 619--671.

\bibitem[P\'olya(1930)]{Polya1930}
G.~P\'olya.
\newblock Sur quelques points de la th\'eorie des probabilit\'es.
\newblock \emph{Ann.\ Inst.\ H.\ Poincar\'e} \textbf1 (1930),  117--161.

\bibitem[R{\'a}th and Szak{\'a}cs(2012)]{RS2012}
B.~R\'{a}th and L.~Szak\'{a}cs.
\newblock Multigraph limit of the dense configuration model and the preferential attachment graph.
\newblock {\em Acta Math.\ Hungar.} {\bf 136} (2012), 196--221.

\bibitem[Riordan(2012)]{OR2012}
O.~Riordan.
\newblock The phase transition in the configuration model.
\newblock {\em Combin.\ Probab.\ Comput.} {\bf 21} (2012), 265--299.

\bibitem{RWEP}
O.~Riordan and L.~Warnke.
\newblock Explosive percolation is continuous.
\newblock {\em Science} {\bf 333} (2011), 322--324.

\bibitem{RWAP}
O.~Riordan and L.~Warnke.
\newblock Achlioptas process phase transitions are continuous.
\newblock {\em Ann.\ Appl.\ Probab.} {\bf 22} (2012), 1450--1464.

\bibitem{RWapsubcr}
O.~Riordan and L.~Warnke.
\newblock The evolution of subcritical {A}chlioptas processes.
\newblock {\em Rand.\ Struct.\ \& Algor.} {\bf 47} (2015), 174--203.

\bibitem{RWapbsr}
O.~{Riordan} and L.~{Warnke}.
\newblock The phase transition in bounded-size Achlioptas processes.
\newblock Preprint~(2017). \texttt{arXiv:1704.08714}.

\bibitem{RWd}
A.~Ruci{\'n}ski and N.C.~Wormald.
\newblock Random graph processes with degree restrictions.
\newblock {\em Combin.\ Probab.\ Comput.} {\bf 1} (1992), 169--180.

\bibitem{Samalam2012}
V.K.~Samalam.
\newblock Preferential attachment alone is not sufficient to generate scale free random networks.
\newblock Preprint~(2012). \texttt{arXiv:1202.1498}.

\bibitem{SW}
J.~Spencer and N.C.~Wormald. 
\newblock Birth control for giants. 
\newblock {\em Combinatorica} {\bf 27} (2007), 587--628.

\bibitem[Steger and Wormald(1999)]{SW1999}
A.~Steger and N.C.~Wormald.
\newblock Generating random regular graphs quickly.
\newblock {\em Combin.\ Probab.\ Comput.} {\bf 8} (1999), 377--396.

\bibitem{T89}
S.~Tavar{\'e}.
\newblock The genealogy of the birth, death, and immigration process.
\newblock In {\em Mathematical Evolutionary Theory}, 
\newblock Princeton University Press (1989), pp.~41--56.

\bibitem{W2019}
L.~Warnke.
\newblock On Wormald's differential equation method. 
\newblock Manuscript~(2019).

\bibitem{WWd}
L.~Warnke and N.~Wormald.
\newblock In preparation~(2019).

\bibitem{Wormald1995DEM}
N.C.~Wormald. 
\newblock Differential equations for random processes and random graphs.
\newblock {\em Ann.\ Appl.\ Probab.} {\bf 5} (1995), 1217--1235.

\bibitem{Wormald1999DEM}
N.C.~Wormald. 
\newblock The differential equation method for random graph processes and greedy algorithms. 
\newblock In {\em Lectures on approximation and randomized algorithms}, pp.~73--155. PWN, Warsaw (1999).


\end{thebibliography}

\normalsize

\begin{appendix}

\section{Appendix}

\subsection{Heuristic for the phase transition location~$\mc$}\label{sec:mc:heuristic}
In this appendix we give an informal explanation for why~$\mc=\frac{n}{2(1+\alpha^{-1})}$ 
should be the phase transition location, 
based on the widely-used heuristic~\cite{SW,KRBN2010,RWapsubcr,JW2016,WWd} 
which predicts~$\mc$ as the blow-up point of the \emph{susceptibility}~$S(m)=S(G^{\alpha}_{n,m}) := \sum_j |C_j|^2/n$,
where~$C_1, C_2, \ldots$ denote the components of~$G^{\alpha}_{n,m}$.  
For simplicity, we henceforth assume that there is a deterministic~approximation 
\begin{equation}\label{def:S}
S(m) \approx s(t) \quad \text{with} \quad t=t(m) := m/n.
\end{equation}
By equation~\eqref{eq:tree:rate} in~\refS{sec:comments} any two distinct tree components~$C_j,C_k$ 
merge with probability proportional to~$[(2+\alpha)|C_j|-2] \cdot [(2+\alpha)|C_k|-2]$, 
in which case the susceptibility changes by $(|C_j|+|C_k|)^2/n-(|C_j|^2+|C_k|^2)/n=2|C_j||C_k|/n$. 
Pretending that nearly all components are not-too-big tree-components 
(which in analogy with the Erd{\H{o}}s--R{\'e}nyi process seems reasonable, 
since here we are only interested in steps leading up to the critical point, 
i.e., before the giant emerges), 
with some hand-waving we loosely expect~that 
\begin{equation}\label{eq:S:changes}
\begin{split}
\E\bigpar{S(m+1)-S(m)\;\big|\;S(m)} 
& \approx \sum_{j \neq k}\lrpar{\frac{[(2+\alpha)|C_j|-2] \cdot [(2+\alpha)|C_k|-2]}{\sum_{\text{$v \neq w$:~non-adj.}} (d_v+\alpha)(d_w+\alpha)} \cdot \frac{2|C_j||C_k|}{n}}\\
& \approx \frac{2}{n} \cdot \frac{\Bigpar{(2+\alpha)\sum_j |C_j|^2-2 \sum_j |C_j|}^2}{\Bigpar{\sum_{v}(d_v+\alpha)}^2} 
= \frac{2}{n} \biggpar{\frac{(2+\alpha)S(m)-2}{2i/n+\alpha}}^2 .
\end{split}
\end{equation}
Inserting~\eqref{def:S}, this suggests (together with~$S(m+1)-S(m)\approx s'(t)/n$ and~$S(0)=1$) 
the differential equation
\begin{align}\label{eq:s:DE}
  s'(t) = 2\biggpar{\frac{(2+\alpha)s(t)-2}{2t+\alpha}}^2 \quad \text{ and } \quad s(0)=1 .
\end{align}
The solution~$s(t)=\frac{\alpha-2t}{\alpha-2(\alpha+1)t}$ blows up at time~$\tc := \frac{1}{2(1+\alpha^{-1})}$, 
so the described heuristic indeed predicts the critical point~$\mc=\tc n$. 
(For comparison, in the Erd{\H{o}}s--R{\'e}nyi process the critical point is~$n/2$  
and its susceptibility blows up at time~$1/2$, since we analogously arrive at~$s'(t)=2s(t)^2$ and~$s(t)=1/(1-2t)$.) 

One of the arguments in~\cite{BNK} is essentially equivalent to this heuristic
(using a different time scale, see~\refS{sec:BNK:pw}). 
The assumed approximation~\eqref{def:S} can be justified rigorously, 
e.g., using either 
the differential equation method~\cite{Wormald1995DEM,Wormald1999DEM,W2019},
or the configuration model transfer method from \refS{sec:other}.

\subsection{Compatibility with previous work}\label{sec:previous} 
In this appendix we show that our giant component results are compatible with previous work~\cite{Pittel,BNK}.

\subsubsection{Work of Pittel (On a random graph evolving by degrees)}\label{sec:Pittel:pw}
Translating to our notation, Pittel~\cite{Pittel} considers the preferential attachment random graph process 
for \emph{constant}~$\alpha  \in (0,\infty)$.  
For~$c > c_\ga := 1/(1+\alpha^{-1})$ he defines~$c^*$ as the (unique) root~$x
\in (0,c_\ga)$ of
\begin{equation}\label{def:pittel:x}
\frac{x\ga^{\ga+1}}{(\ga+x)^{\ga+2}}=\frac{c \ga^{\ga+1}}{(\ga+c)^{\ga+2}} .
\end{equation}
Assuming~$n^{1/4}(c-c_\ga) \to \infty$, then~\cite[Theorem~1]{Pittel} ensures that the largest component has size 
\begin{equation}\label{eq:L1:Pittel}
L_1(cn/2) = \left[1-\left(\frac{\ga+c^*}{\ga+c}\right)^\ga\right] n \cdot (1+o_p(1)).
\end{equation}
Using the following auxiliary claim, 
Theorem~\ref{thm:main1} implies~\eqref{eq:L1:Pittel} under the weaker assumption~$n^{1/3}(c-c_\ga) \to \infty$ 
conjectured in~\cite{Pittel} 
(since for~$m := \mc (1+\eps)$ we have~$c=(1+\eps)c_\ga$ and~$n^{1/3}\eps = \Theta(n^{1/3}(c-c_\ga)) \to \infty$). 
\begin{claim}
Let~$\eps>0$. Then for~$c:=(1+\eps)c_\ga$ we have~$\rho_{\ga}(\eps)=1-\left(\frac{\ga+c^*}{\ga+c}\right)^\ga$. 
\end{claim}
\begin{proof}
In view of~\eqref{def:pittel:x} let 
\begin{equation}\label{eq:xi:def}
\xi := \frac{c^*(\ga+c)}{c(\ga+c^*)} = \left(\frac{\ga+c^*}{\ga+c}\right)^{\ga+1} ,
\end{equation}
and observe that $\xi \in (0,1)$ holds (since~$\ga \in (0,\infty)$
and~$c^*<c_\ga< c$).
Noting that 
\begin{align}
  1-\xi = \frac{c(\ga+c^*)-c^*(\ga+c)}{c(\ga+c^*)} =
\frac{\ga(c-c^*)}{c(\ga+c^*)} , 
\end{align}
we infer
\begin{align}
    \frac{\ga+c}{\ga+c^*} = 1 + \frac{c-c^*}{\ga+c^*}
  = 1 + \frac{c}{\ga} \cdot (1-\xi)
  =1+ \frac{1+\eps}{\ga+1}(1-\xi). .
\end{align}
Inserting this back into~\eqref{eq:xi:def}, a comparison with the first equation in~\eqref{eq:NB:xi} 
now shows that~$\xi$ is indeed the unique solution in $(0,1)$
to~$\E D\xi^{D-1} =\xi \E D$,  
and thus, by~\eqref{emma}
and \eqref{eq:xi:def}, we obtain~$\rho_{\ga}(\eps) =
1-\xi^{\ga/(\ga+1)} =
1-\left(\frac{\ga+c^*}{\ga+c}\right)^\ga$.
\end{proof}

\subsubsection{Work of Ben-Naim and Krapivsky (Popularity-driven networking)}\label{sec:BNK:pw}
Ben-Naim and Krapivsky~\cite{BNK} investigate 
the special case $\alpha=1$ of the preferential attachment random multigraph process 
(using a different time-parametrization) via the kinetic theory methodology from statistical physics, 
also called rate equation approach. 
Inspecting equation~(7) in~\cite{BNK} (solving~$2m/n=\langle j \rangle =t/(1-t)$ for~$t$), 
for fixed~$\eps>0$ it follows that adding~$m=(1+\eps)n/4$ edges corresponds to their time
\begin{equation}\label{eq:time:BNK}
t:=(1+n/2m)^{-1} = 1/3 + 2\eps/9 + O(\eps^2).
\end{equation}
The discussion of the function~$g(t)$ in~\cite[p.~4, together with~(1)]{BNK}   
then says that 
the asymptotic fraction of vertices in the giant component should be approximately equal to
\begin{align}
  3(t-1/3) \sim 2\eps/3 \sim \rho_1(\eps)    \qquad\text{as }\eps \searrow 0,
\end{align}
which is made rigorous by
Theorems~\ref{thm:main1} and~\ref{Tmulti}.
%
Finally, noting that~\eqref{eq:time:BNK} and~\eqref{pn} imply
the identity~$p_n=t$ (when~$\alpha=1$), \refT{thm:degree} also 
gives a rigorous version of the asymptotic degree distribution~(9) 
in~\cite{BNK}.

\end{appendix}

\end{document}